\documentclass[11pt]{amsart}
\usepackage{amssymb, latexsym, mathrsfs,   color }
\usepackage[all]{xy}
\usepackage[colorlinks=true, pdfstartview=FitV,linkcolor=blue,citecolor=blue,urlcolor=blue]{hyperref}

    \advance \topmargin by -\headheight
    \advance \topmargin by -\headsep

    \evensidemargin \oddsidemargin
\newtheorem {theorem}    {Theorem}[section]

\newtheorem {lemma}      [theorem]    {Lemma}
\newtheorem {corollary}  [theorem]    {Corollary}
\newtheorem {proposition}[theorem]    {Proposition}

\newcommand{\bb}{\mathbb}
\renewcommand{\rm}{\mathrm}

\newcommand{\cal}{\mathcal}
\newcommand{\GG}{\mathrm{G}}
\newcommand{\GGL}{\mathrm{GL}}
\newcommand{\UU}{\mathrm{U}}

\newcommand{\Fq}{\bb{F}_q}
\newcommand{\fq}{(\bb{F}_q)}
\theoremstyle{definition}

\newtheorem{ex}[theorem]{Example}

\numberwithin{equation}{section}

\begin{document}

\title{On the Gan-Gross-Prasad problem for  finite unitary groups}

\date{\today}

\author[Dongwen Liu]{Dongwen Liu}

\address{School of Mathematical Science, Zhejiang University, Hangzhou 310027, Zhejiang, P.R. China}

\email{maliu@zju.edu.cn}

\author[Zhicheng Wang]{Zhicheng Wang}

\address{School of Mathematical Science, Zhejiang University, Hangzhou 310027, Zhejiang, P.R. China}

\email{11735009@zju.edu.cn}

\begin{abstract}
In this paper we study the Gan-Gross-Prasad  problem for unitary groups over finite fields. Our results provide complete answers for unipotent representations, and we obtain the explicit branching of these representations.\end{abstract}

\maketitle

\section{Introduction}

In \cite{GP1, GP2}, B. Gross and D. Prasad studied the restriction problem for special orthogonal groups over a local field and formulated a number of conjectures. Joint with W.T. Gan, in \cite{GGP1} they extended the conjecture to all classical groups, which are nowadays known as the  local Gan-Gross-Prasad conjecture. To be a little more precise, the multiplicity one property holds in this situation, namely for a relevant pair of classical groups $G\supset H$ and their irreducible admissible  representations $\pi$ and $\sigma$ respectively,
\[
m(\pi,\sigma):=\dim\rm{Hom}_H(\pi,\sigma)\leq 1;
\]
and the invariant attached to $\pi$ and $\sigma$ that detects the nonvanishing of the multiplicity $m(\pi,\sigma)$ is the local root number associated to their Langlands parameters, which are assumed to be generic.  In the $p$-adic case, the local Gan-Gross-Prasad conjecture  has been resolved by J.-L. Waldspurger and C. M\oe glin and J.-L. Waldspurger \cite{W1, W2, W3, MW} for orthogonal groups, by R. Beuzart-Plessis \cite{BP1, BP2} and W. T. Gan and A. Ichino \cite{GI} for unitary groups, and by H. Atobe \cite{Ato} for symplectic-metaplectic groups.

The main goal of this paper is to study the Gan-Gross-Prasad problem for unipotent representation of finite unitary groups. To begin with, we first set up some notations. Let $\overline{\mathbb{F}}_q$ be an algebraic closure of a finite field $\mathbb{F}_q$, which is of characteristic $p>2$.  Let $G=\UU_n$ be an $\bb{F}_q$-rational form of $\GGL_n(\overline{\mathbb{F}}_q)$, and $F$ be the corresponding Frobenius endomorphism, such that the group of $\Fq$-rational points  is
$G^F=\UU_n(\Fq)$.
Let $Z$ be the center of $G^F$. We will assume that $q$ is large enough such that the main theorem in \cite{S2} holds, namely assume that
 \begin{itemize}
 \item $q$ is large enough such that $T^F/Z$ has at least two Weyl group orbits of regular characters, for every $F$-stable maximal torus $T$ of $G$.
 \end{itemize}
For an $F$-stable maximal torus $T$ of $G$ and a character $\theta$ of $T^F$,  let $R_{T,\theta}^G$ be the virtual character of $G^F$ defined by P. Deligne and G. Lusztig in \cite{DL}. An irreducible representation $\pi$ is called unipotent if there is an $F$-stable maximal torus $T$ of $G$ such that $\pi$ appears in $R_{T,1}^G$. For two representations $\pi$ and $\pi'$ of a  finite group $H$, define
\[
\langle \pi,\pi'\rangle_H:=\dim \mathrm{Hom}_H(\pi,\pi').
\]
%In order not to cause confusion, we denote $R_{T,\theta}^G$ by $R_{T,\theta}^{\GGL_n(\Fq)}$ (or $R_{T,\theta}^{\UU_n(\Fq)}$) if $G^F=\GGL_n(\Fq)$ (or $\UU_n(\Fq)$).
% We simply write $\langle \pi, \pi'\rangle$ for  $\langle \pi,\pi'\rangle_H$ when no confusion arises.

Let $\pi$ and $\pi'$ be  irreducible representations of $\UU_{n}(\Fq)$ and $\UU_{m}(\Fq)$ respectively, where $n\ge m$.
%Let $P$ be a maximal parabolic subgroup of $U_{n+1}(q)$ with Levi factor $GL_{r}(q^2) \times U_{m}(q)$ (so that $m + 2\ell = n + 1$) and let $\tau$ be a cuspidal representation of $GL_{r}(q^2)$.
The Gan-Gross-Prasad problem is concerned with the multiplicity
\[
m(\pi, \pi'):=\langle \pi\otimes\bar{\nu},\pi' \rangle_{H(\Fq)} = \dim \mathrm{Hom}_{H(\Fq)}(\pi\otimes\bar{\nu},\pi' )
\]
where the data $(H, \nu)$ is defined as in \cite{GGP1} (see  \cite{LW} for details in this case). According to whether $n-m$ is odd or even, the above Hom space is called the Bessel model or Fourier-Jacobi model. In \cite[Proposition 5.3]{GGP2}, it was shown for the Bessel case  that if both $\pi$ and $\pi'$ are cuspidal, then
\[
m(\pi,\pi') \le 1.
\]
Our formulation of the models differs slightly from that in the Gan-Gross-Prasad conjecture \cite{GGP1}, up to taking the contragradient of $\pi'$. 
This is more convenient for our discussion, which will be clear from the context below. 
On the other hand, in this paper we focus on unipotent representations of $\UU_{n}(\Fq)$, which are self-dual (c.f. \cite[Proposition 6.6]{LW}) and thus for $\pi$ unipotent the above Hom space coincides with $\rm{Hom}_{H(\Fq)}(\pi\otimes\pi',\nu)$.

Recall from \cite{LS} that irreducible unipotent representations of $\UU_n(\Fq)$ are parameterized by irreducible representations of $S_n$. It is well-known that  the latter are  parameterized by partitions of $n$.   For a partition $\lambda$ of $n$, denote by $\pi_\lambda$ the  corresponding unipotent representation of $\UU_n(\Fq)$.  As is standard, we realize partitions  as Young diagrams, and denote by ${}^t\lambda$ the transpose of $\lambda$.  In \cite{AMR}, a notion of {\sl 2-transverse} for two partitions  was introduced, which will be recalled in details in Section \ref{ssec3.2}.

Our first main result is the following.
\begin{theorem}\label{1.1}
Assume that $n \ge m$. Let $\lambda$ and $\lambda'$ be partitions of $n$ and $m$ respectively. Then
\[
m(\pi_\lambda, \pi_{\lambda'})=\left\{
\begin{array}{ll}
1, &  \textrm{if }\ \lambda \textrm{ and } \lambda' \textrm{ are }  2\textrm{-transverse},\\
0, & \textrm{otherwise.}
\end{array}\right.
\]
\end{theorem}

It is interesting to notify the connection beween the above branching rule and the theta correspondence. This connection can be built via the so-called Alvis-Curtis dual, which will be recalled in Section \ref{ssec2.1}. Let $\mu$ and $\mu'$ be partitions of $n$ and $m$ respectively, and $\omega_{n,m}$ be the Weil representation of $\UU_n\fq\times\UU_m\fq$ (see \cite{S2} for details). In \cite{AMR}, it was shown that the theta correspondence between unipotent representations is given by 
\[
\langle\pi_{\mu}\otimes\pi_{\mu'},\omega_{n,m}\rangle_{\UU_n\fq\times\UU_m\fq}=\left\{
\begin{array}{ll}
1, &  \textrm{if }\ {}^t\mu \textrm{ and } {}^t\mu' \textrm{ are }  2\textrm{-transverse},\\
0, & \textrm{otherwise.}
\end{array}\right.
\]
On the other hand, the Alvis-Curtis dual of $\pi_\lambda$ is known to be $\pi_{{}^t\lambda}$. Combining these facts, 
 the content of Theorem \ref{1.1} can be visualized as a diagram
\[
\setlength{\unitlength}{0.8cm}
\begin{picture}(20,5)
\thicklines
\put(5.8,4){$\pi_\lambda$}
\put(7.3,4.2){unipotent part of GGP}
\put(5.8,1){$\pi_{^t\lambda}$}
\put(13.1,4){$\bigoplus_{\lambda'}\pi_{\lambda'}$}
\put(13,1){$\bigoplus_{^t\lambda'}\pi_{^t\lambda'}$}
\put(8.4,1.3){Theta lifting}
\put(6,3.6){\vector(0,-1){2.1}}
\put(13.6,3.6){\vector(0,-1){2.1}}
\put(6.5,1.1){\vector(2,0){6.2}}
\put(6.5,4){\vector(2,0){6.2}}
\end{picture}
\]
where the vertical arrows stand for taking the Alvis-Curtis dual.

In special cases, this result overlaps with our previous work \cite{LW} on the descent problem for finite unitary groups. However, we have different point of views, and the main results are to some extent complementary to each other.

We will only prove an equivalent form of Theorem \ref{1.1} for the Bessel case;  the proof for the Fourier-Jacobi cases is similar and will be omitted.
Let us outline the strategy of the proof. First of all, Proposition \ref{7.21} and Proposition \ref{7.31} show that parabolic induction preserves multiplicities, which are finite field analogs of Theorem 15.1 and Theorem 16.1 in \cite{GGP1} respectively for unipotent representations. This reduces the calculation to the basic case. For the Bessel case,  in order to compute the right hand side of the equation
\[
m(\pi,\pi')=\langle \pi\otimes\bar{\nu}, \pi'\rangle _{H(\Fq)}=\langle R^{\UU_{n+1}}_{L}(\tau\otimes\pi'),\pi\rangle _{\UU_{n}(\Fq)}
\]
in Proposition \ref{7.21}, we shall reduce the index $n$ by using see-saw dual pairs. This will prove Theorem \ref{1.1} by induction on $n$. To apply the see-saw arguments, we need the explicit theta correspondence of unipotent representations of finite unitary groups, which is given in \cite{AMR}.

By Theorem \ref {1.1}, for a fixed unipotent representation $\pi$ of $\UU_n\fq$, we have an explicit description of the mulitplicities $m(\pi, \pi')$ for unipotent representations $\pi'$ of $\UU_m\fq$ with $m\leq n$. Our next goal is to describe $m(\pi,\pi')$ for an arbitrary representation $\pi'$ of $\UU_m\fq$. Our main tools are the Lusztig correspondence \cite{L} and Reeder's branching formula introduced in \cite{R} (c.f. \cite{LW}).

Recall that for $G^F=\UU_n\fq$ one has the dual group $G^{*F}=\UU_n\fq$. For a semisimple element $s \in G^{*F}$, we say  that $1\notin s$ if $1$ is not an eigenvalue of $s$.
Suppose that $s \in  \UU_n\fq$ is semisimple and conjugate to $\rm{diag}(s',1_{n-m})$ where $s' \in \UU_m\fq$ is semisimple and $1\notin s'$.
Let $P=LV$ be a parabolic subgroup of $\UU_n$ such that $L$ is $F$-stable and $s\in L^F\cong\UU_m\fq\times\UU_{n-m}\fq$, but $P$ is not necessarily $F$-stable. For each $\pi$ in the Lusztig series $\mathcal{E}(\UU_n\fq,s)$, by the Lusztig correspondence there exist unique $\pi'\in \mathcal{E}(\UU_m\fq,s')$ and $\pi_\lambda\in\mathcal{E}(\UU_{n-m}\fq,1)$ with $\lambda$ a partition of $n-m$ such that
\[
\pi=\pm R^{\UU_n}_{L}(\pi'\otimes\pi_\lambda),
\]
 where $R^{\UU_n}_{L}(\pi'\otimes\pi_\lambda)$ is the virtual representation defined by Deligne and Lusztig. Note that every irreducible representation of $\UU_n\fq$ is of this form. By abuse of notation, below we suppress the sign and simply denote by $R^{\UU_n}_{L}(\pi'\otimes\pi_\lambda)$ the irreducible representation.

Then our second main result is the following.

 \begin{theorem}\label{1.2}
Let $\lambda$ and $\lambda'$ be partitions of $n$ and $m$ respectively, $m\leq n$.
Let $\pi\in \mathcal{E}(\UU_\ell\fq, s)$ with $\ell+m\le n+1$ and $1\not\in s$. Then
\[
m(\pi_{\lambda},R_{\UU_{\ell}\times\UU_m}^{\UU_{\ell+m}}(\pi\otimes \pi_{\lambda'}) )=\left\{
\begin{array}{ll}
1, &  \textrm{if }\ \lambda \textrm{ and } \lambda' \textrm{ are }  2\textrm{-transverse and } \pi=\pi^{reg}_s,\\
0, & \textrm{otherwise,}
\end{array}\right.
\]
where $\pi_s^{reg}$ is the unique regular character in $\mathcal{E}(\UU_{\ell}\fq,s)$.
\end{theorem}

It will be interesting to isolate the so-called basic case that $n-m=1$ or $0$.  Then Theorem \ref{1.2} gives us the following explicit spectral decompositions,  which extends  \cite[Theorem 1.2 and Theorem 1.4]{HZ} in the case of finite unitary groups.

 \begin{corollary}\label{1.3}
Let $\lambda$ be a partition of $n$. Then the following hold.

(i) $\pi_\lambda|_{\UU_{n-1}\fq}$ has the multiplicity-free decomposition
\[
\pi_{\lambda}|_{\UU_{n-1}\fq}=\bigoplus_{\lambda', s} R^{\UU_{n-1}}_{\UU_{n-1-m}\times\UU_m}(\pi^{reg}_s\otimes\pi_{\lambda'}),
\]
where the sum runs over partitions $\lambda'$ of $m<n$ such that $\lambda$ and $\lambda'$ are $2$-transverse, and semisimple conjugacy classes of $s\in\UU_{n-1-m}\fq$ such that $1\not\in s$.

(ii) Let $\omega_n$ be the Weil representation of $\UU_n\fq$. Then $\pi_{\lambda}\otimes\omega_{n}$ has the multiplicity-free decomposition
\[
\pi_{\lambda}\otimes\omega_{n}=\bigoplus_{\lambda',s} R^{\UU_n}_{\UU_n\times\UU_{n-m}}(\pi^{reg}_s\otimes\pi_{\lambda'}),
\]
where the sum runs over partitions $\lambda'$ of $m\le n$ such that $\lambda$ and $\lambda'$ are $2$-transverse, and  semisimple conjugacy classes of $s\in \UU_{n-m}\fq$ such that $1\not\in s$.
\end{corollary}

It is not surprising that the branching rules for the Bessel case and the Fourier-Jacobi case in Corollary \ref{1.3} look very similar. Indeed these two models are closely related to each other by a see-saw diagram.  One can also compare with the $p$-adic case, for which the branching rules for both models are governed by a distinguished pair of characters of the component groups of L-parameters, that are prescribed by the local Gan-Gross-Prasad conjecture using local root numbers. 

Finally we have a few remarks for the Bessel case about the assumption on $\bb{F}_q$.
\begin{itemize}

\item Proposition \ref{7.21} holds without the assumption on $q$ in \cite{S2}. In other words, 
\[
m(\pi,\pi')=\langle \pi\otimes\bar{\nu}, \pi'\rangle _{H(\Fq)}=\langle R^{\UU_{n+1}}_{L}(\tau\otimes\pi'),\pi\rangle _{\UU_{n}(\Fq)}
\] 
holds for any $\bb{F}_q$ with $q$ odd.

\item Since any irreducible representation of $\UU_n\fq$ is uniform, we can calculate the right hand side of the above equation using Reeder's branching formula introduced in \cite{R} (c.f. \cite{LW}), which asserts that the multiplicity is a polynomial of $q$. Moreover, by Proposition \ref{proposition:A}, the multiplicity is a constant.

\item The multiplicity in the Bessel case of Theorem \ref{1.2} is a constant if $q$ is large enough. 
\end{itemize}
It follows that result for the Bessel case in Theorem \ref{1.2} holds for any $\bb{F}_q$ with $q$ odd.

 This paper is organized as follows. In Section \ref{sec2}, we briefly recall the theory of Deligne-Lusztig characters and classification of representation of finite unitary groups. In Section \ref{sec3}, we recall the theory of Weil representation, theta correspondence and see-saw dual pairs. In Section \ref{sec4} we prove Theorem \ref{1.1}. In Section \ref{sec5} we  prove Theorem \ref{1.2}.

\section{Deligne-Lusztig characters} \label{sec2}
Let $G$ be a connected reductive algebraic
group over $\mathbb{F}_q$. In \cite{DL}, P. Deligne and G. Lusztig defined a virtual character $R^{G}_{T,\theta}$ of $G^F$, associated to an $F$-stable maximal torus $T$ of $G$ and a character $\theta$ of $T^F$. The construction of Deligne-Lusztig characters makes use of the theory of $\ell$-adic cohomology. We only review some standard facts on these characters (cf. \cite[Chapter 7]{C}), which will be used in this paper.

More generally, let $L$ be an $F$-stable Levi subgroup of a parabolic subgroup $P$ which is not necessarily $F$-stable, and $\pi$ be a representation of the group $L^F$. Then $R^G_L(\pi)$ is a virtual character of $G^F$. If $P$ is $F$-stable, then the Deligne-Lusztig induction coincides with the parabolic induction
\[
R^G_L(\pi)= \rm{Ind}^{G^F}_{P^F}(\pi).
\]
For example if $T$ is contained in an $F$-stable Borel subgroup $B$, then
\[
R^G_{T,\theta}=\rm{Ind}^{G^F}_{B^F}\theta.
\]
 In general, if $y = su$ is the Jordan decomposition of an element $y \in G^F$, then
\begin{equation}\label{dl}
R^{G}_{T,\theta}(y)=\frac{1}{|C^{0}(s)^F|}\sum_{g\in G,s^g \in T}\theta(s^{g})Q^{C^{0}(s)}_{{}^{g}T}(u)
\end{equation}
where $C^{0}(s) = C^{0}_{G}(s)$ is the connected component of the centralizer of $s$ in $G$, and $Q^{C^{0}(s)}_{{}^{g}T} = R^{ C^{0}(s)}_{{}^{g}T,1}(u)$ is the Green function of $C^{0}(s)$ associated to ${}^{g}T$. Note that $s^{g}=g^{-1}sg \in T$ if and only if ${}^{g}T=gTg^{-1} \subset C^{0}(s)$.

\begin{ex}
$\UU_n\times \UU_m$ can be embedded as an $F$-stable Levi subgroup of $\UU_{n+m}$, which is not a Levi subgroup of any $F$-stable parabolic subgroup.
\end{ex}

The following facts are standard.

\begin{proposition}[Induction in stages] \label {3.1}
 Let $Q \subset P$ be two parabolic subgroups of $G$, with $F$-stable Levi subgroups $M\subset L$ respectively. Then
\[
R^G_L \circ R_M^L  = R_M^G.
\]
\end{proposition}

\begin{proposition}[Weak orthogonality] \label {3.2}
Let $T_1$ and $T_2$ be two $F$-stable maximal tori of $G$. Set
\[
N_G(T_1,T_2) = \{g \in G|{^gT_1} = T_2\},
\]
 and $W_G(T_1,T_2):=T_1\backslash N_G(T_1,T_2)\cong  N_G(T_1,T_2)/T_2$. Then
\[
\langle R^G_{T_1,\theta_1},R^G_{T_2,\theta_2}\rangle_{G^F}=\#\{w\in W_G(T_1,T_2)^F| {^wT_1}=T_2\textrm{ and } {^w\theta_1}=\theta_2\}.
\]
In particular, if $T_1$ and $T_2$ are not $G^F$-conjugate, then
$
\langle R^G_{T_1,\theta_1},R^G_{T_2,\theta_2}\rangle_{G^F}=0;
$
and
\[
\langle R^G_{T,\theta},R^G_{T,\theta}\rangle_{G^F}=|W_{T}(\theta)^F|,
\]
where
\[
W_{T}(\theta)=\{w\in W_G(T)  :  {}^w\theta=\theta\}.
\]
\end{proposition}

\subsection{Unipotent representations and duality}\label{ssec2.1}
The classification of the representations of $\UU_n(\Fq)$ was given by Lusztig and Srinivasan in \cite{LS}. Denote by $W_n\cong S_n$ the Weyl group of the diagonal torus in  $\UU_n(\Fq)$.

\begin{theorem}\label{thm3.3}
Let $\sigma$ be an irreducible representation of $S_n$. Then
\[
R_\sigma^{\UU_n}:=\frac{1}{|W_n|}\sum_{w\in W_n}\sigma(ww_0)R_{T_w,1}^{\UU_n}
\]
is (up to sign) a unipotent representation of $\UU_n(\Fq)$ and all unipotent representations of $\UU_n(\Fq)$ arise in this way.\end{theorem}

It is well-known that irreducible representations of $S_n$ are parametrized by partitions of $n$. For a partition $\lambda$ of $n$,  denote by  $\sigma_\lambda$ the corresponding representation of $S_n$, and
let $\pi_\lambda= \pm R_{\sigma_\lambda}^{\UU_n}$ be the corresponding unipotent representation of $\UU_n(\Fq)$.
By Lusztig's result \cite{L},  $\pi_\lambda$ is (up to sign) a unipotent cuspidal representation  of $\UU_n(\Fq)$ if and only if $n=\frac{k(k+1)}{2}$ for some positive integer $k$ and  $\lambda=[k,k-1,\cdots,1]$.

For a character $\chi$ of $G^F$, denote by $\chi^*$ the Alvis-Curtis dual of $\chi$ defined  in e.g. \cite{A, Cu, K}. If $\chi$ is an irreducible character of $G^F$, then $\chi^*$ is (up to sign) an irreducible character of $G^F$ as well.
By \cite[Proposition 9.3.4]{C},
\[
(\varepsilon_G\varepsilon_T R^G_{T,\theta})^*=R^G_{T,\theta},
\]
where $\varepsilon_G=(-1)^{\rm{rk}(G)}$. Thus
\[
(R_\sigma^{\UU_n})^*=\frac{1}{|W_n|}\sum_{w\in W_n}\varepsilon_{\UU_n}\varepsilon_{T_w} \sigma(ww_0)R_{T_w,1}^{\UU_n}.
\]
It is well-known that for a character $\sigma_{\lambda}$ of $S_n$ corresponding to a partition $\lambda$ of $n$, $\sigma_{\lambda}\otimes\rm{sgn}\cong \sigma_{^t\lambda}$, hence up to sign
\begin{equation}\label{dual}
\pi_{\lambda}^*\cong\pi_{^t\lambda}.
\end{equation}

\subsection{Regular characters}

An $F$-stable maximal torus $T$ is said to be minisotropic if $T$ is not contained in any $F$-stable proper parabolic subgroup of $G$. Then a representation $\pi$ of $G^F$ is cuspidal if and only if
\[
\langle \pi,R^G_{T,\theta}\rangle_{G^F}=0
\]
whenever $T$ is not minisotropic, for any character $\theta$ of $T^F$ (see \cite[Theorem 6.25]{S1}). Note that if $G^F=\GGL_n(\Fq)$, then $T$ is said to be minisotropic when $T^F\cong \GGL_1(\mathbb{F}_{q^n})$.

Assume that $\theta\in \widehat{T^F}$, $\theta'\in \widehat{T'^F}$ where $T$, $T'$ are $F$-stable maximal tori. The pairs $(T,\theta)$, $(T',\theta')$ are said to be geometrically conjugate if for some $n\ge 1$, there exists $x \in G^{F^n}$ such that
\[
^xT^{F^n} = T ^{\prime F^n}\ \mathrm{and}\ \ ^x(\theta \circ  N^T_n) = \theta'\circ N^{T'}_n
 \]
 where $N_n^T: T^{F^n}\to T^F$ is the norm map. By \cite[p. 378]{C}, for any geometrically conjugate class $\kappa$, there is a unique regular character $\pi^{reg}_\kappa$ appearing in $R^G_{T,\theta}$ for some $(T,\theta)\in \kappa$; and any regular character appears in exactly one geometric conjugacy class. Moreover
\begin{equation}\label{reg}
\pi^{reg}_\kappa=\sum_{(T,\theta)\in\kappa \ \rm{mod} \ G^F}\frac{\varepsilon_G \varepsilon_T R^G_{T,\theta}}{\langle R^G_{T,\theta},R^G_{T,\theta}\rangle _{G^F}}.
\end{equation}
The above equation implies that $\pi_\kappa^{reg}$ appears in $R^G_{T,\theta}$ for every pair $(T,\theta)\in \kappa$. Thus $\pi_\kappa^{reg}$ is cuspidal if and only if $T$ is minisotropic and $\theta$ is regular for every pair $(T,\theta)\in \kappa$.  Here $\theta$ regular means that
\[
^x\theta=\theta, \ x\in W_G(T)^F \ \mathrm{if \ and \ only\ if\ } x=1.
\]
In particular, if $\tau$ is an irreducible cuspidal representation of $\GGL_n(\Fq)$, then there is a pair $(T,\theta)$ with $T$ an $F$-stable minisotropic maximal torus and $\theta$ regular such that
$
\tau=\pm R_{T,\theta}^G.
$

\subsection{Lusztig correspondence}

Let $G^*$ be the dual group of $G$. We still denote the Frobenius endomorphism of $G^*$ by $F$, and $G^{*F}$ the group of rational points. It is known that there is a bijection between the set of $G^F$-conjugacy classes of $(T, \theta)$ and the set of $G^{*F}$-conjugacy classes of $(T^*, s)$ where $T^*$ is a $F$-stable maximal torus in $G^*$ and $s \in   T^{*F}$ . If $(T, \theta)$ corresponds to $(T^*, s)$, then $R_{T,\theta}^G$ will be also denoted by $R_{T^*,s}^G$.
For a semisimple element $s \in G^{*F}$, define
\[
\mathcal{E}(G^F,s) = \{ \chi \in \mathcal{E}(G^F)  :  \langle \chi, R_{T^*,s}^G\rangle \ne 0\textrm{ for some }T^*\textrm{ containing }s \}.
\]
The set $\mathcal{E}(G^F,s)$ is called  the Lusztig series, and it is known that $\mathcal{E}(G^F)$ is partitioned into
Lusztig series indexed by the conjugacy classes $(s)$ of semisimple elements $s$,
i.e.,
\[
\mathcal{E}(G^F)=\coprod_{(s)}\mathcal{E}(G^F,s).
\]

The following result is fundamental for the classification of $\mathcal{E}(G)$:

\begin{proposition}[Lusztig]\label{Lus}
There is a bijection
\[
\mathcal{L}_s:\mathcal{E}(G^F,s)\to \mathcal{E}(C_{G^{*F}}(s),1),
\]
extended by linearity to a map between virtual characters satisfying that
\[
\mathcal{L}_s(\varepsilon_G R^G_{T^*,s})=\varepsilon_{C_{G^{*F}}(s)} R^{C_{G^{*F}}(s)}_{T^*,1}.
\]
\end{proposition}

From now on assume that $G^F=\UU_n\fq$. In this case, $G^{*F}=\UU_n\fq$.
For later use, we prove the following irreducibility result using Lusztig correspondence.

\begin{proposition}\label{irr}
Let $s$ be a semisimple element of $\UU_n\fq$, which is  $\UU_n\fq$-conjugate to $\rm{diag}(s_1,s_2)$ for some semisimple elements $s_1$ and $s_2$ in $\UU_{n_1}\fq$ and $U_{n_2}\fq$ respectively, with $n=n_1+n_2$. Assume that $s_1$ and $s_2$ have no common eigenvalues.  Then for any  $\pi_1\in \mathcal{E}(\UU_{n_1}\fq,s_1)$ and $\pi_2\in\mathcal{E}(\UU_{n_2}\fq,s_2)$, $R^{\UU_n}_{\UU_{n_1}\times\UU_{n_2}}(\pi_1\otimes\pi_2)$ is (up to sign) an irreducible representation. Moreover
\[
R^{\UU_n}_{\UU_{n_1}\times\UU_{n_2}}(\pi_1 \otimes\pi_2) \cong R^{\UU_n}_{\UU_{n_1}\times\UU_{n_2}}(\pi_1'\otimes\pi_2')
\]
if and only if $\pi_1\cong \pi_1'$ and $\pi_2\cong \pi_2'$.
\end{proposition}

\begin{proof} By the assumption on $s_1$ and $s_2$, one has
\[
C_{\UU_n}(s)\cong C_{\UU_{n_1}}(s_1)\times C_{\UU_{n_2}}(s_2).
\]
One may write $\pi_i$ as a linear combination of $R^{\UU_{n_i}}_{T_i^*, s_i}$, where $T_i^*$ runs over $\UU_{n_i}\fq$-conjugacy classes of $F$-stable maximal tori of $\UU_{n_i}$ containing $s_i$. From Proposition \ref{Lus}, it is not hard to see that up to sign
\[
\cal{L}_s\left(R^{\UU_n}_{\UU_{n_1}\times\UU_{n_2}}(\pi_1\otimes\pi_2)\right)\cong \cal{L}_{s_1}(\pi_1)\otimes \cal{L}_{s_2}(\pi_2),
\]
which is an irreducible unipotent representation of $C_{\UU_{n_1}}(s_1)\times C_{\UU_{n_2}}(s_2)$. Hence $R^{\UU_n}_{\UU_{n_1}\times\UU_{n_2}}(\pi_1\otimes\pi_2)$ is up to sign an irreducible representation of $\UU_n\fq$. The last assertion of the Proposition is obvious.
\end{proof}

In \cite[Lemma 6.2]{LW} we proved the following useful special case of Proposition \ref{irr}. Put
\begin{equation}
\GG_\ell:=\mathrm{Res}_{\mathbb{F}_{q^2}/\Fq}\GGL_\ell,
\end{equation}
so that $\GG_\ell(\Fq)=\mathrm{GL}_\ell(\mathbb{F}_{q^2})$.  Let $\tau$ be an irreducible cuspidal representation of $\GG_\ell\fq$ which is not conjugate self-dual. Then $R^{\UU_n}_{\GG_\ell\times\UU_{n-2\ell}}(\tau\otimes\pi_\lambda)$ is irreducible for any unipotent representation $\pi_\lambda$ of $\UU_{n-2\ell}\fq$.

\section{Weil representations and see-saw dual pairs}\label{sec3}

Let $\omega_{\rm{Sp}_{2N}}$ be the character of the Weil representation (cf. \cite{Ger}) of the finite symplectic group $\rm{Sp}_{2N}(\Fq)$, which depends on a  nontrivial additive character $\psi$ of $\Fq$.
%Let $\omega^{\#}_{\rm{Sp}_{2n}}$ denote the uniform projection of $\omega_{\rm{Sp}_{2N}}$, i.e. the projection onto the subspace of virtual characters spanned by all the Deligne-Lusztig characters.
Let $(G, G^{\prime})$ be a reductive dual pair in $\rm{Sp}_{2N}$, and write
$\omega_{G,G'}$ for the restriction of $\omega_{\rm{Sp}_{2N}}$ to $G^F\times G'^F$. Then it decomposes into a direct sum
\[
\omega_{G,G'}=\bigoplus_{\pi,\pi'} m_{\pi,\pi '}\pi\otimes\pi '
\]
where $\pi$ and $\pi '$ run over irreducible representations of $G^F$ and $G'^F$ respectively, and $m_{\pi,\pi'}$ are nonnegative integers.. We can rearrange this decomposition as
\[
\omega_{G,G'}=\bigoplus_{\pi} \pi\otimes\Theta_{G,G'}(\pi )
\]
 where $\Theta_{G, G'}(\pi ) = \bigoplus_{\pi'} m_{\pi,\pi '}\pi '$ is a (not necessarily irreducible) representation of $G'^F$, called the (big) theta lifting of $\pi$ from $G^F$ to $G'^F$. Write $\pi'\subset \Theta_{G,G'}(\pi)$ if $\pi\otimes\pi'$ occurs in $\omega_{G,G'}$, i.e. $m_{\pi, \pi'}\neq 0$.  We remark that even if $\Theta_{G,G'}(\pi)=:\pi'$ is irreducible, one only has
 \[
 \pi\subset \Theta_{G',G}(\pi'),
 \]
 while the equality does not  necessarily hold.

Consider a dual pair of unitary groups $(G,G')=(\UU_n, \UU_{n'})$ in $\rm{Sp}_{2nn'}$. Denote $\omega_{G, G'}$ by $\omega_{n, n'}$, and $\Theta_{G,G'}$ by $\Theta_{n,n'}$. In particular, we denote by $\omega_n$ the restriction of $\omega_{\rm{Sp}_{2n}}$ to $\rm{U}_n(\Fq)$. By \cite[Theorem 3.5]{AM}, theta lifting between unitary groups sends unipotent representations to unipotent representations, and we will recall the explicit correspondence later.

By  \cite[Lemma 6.2 and Proposition 6.4]{LW}, we have the following compatibility for the theta lifting and parabolic induction.

\begin{proposition} \label{theta}
Let $\tau$ be an irreducible cuspidal representation of $\GG_\ell\fq$ which is not conjugate self-dual, $\pi$ be an irreducible unipotent representation of $\UU_n(\Fq)$, and  $\pi':=\Theta_{n,n'}(\pi)$. Then we have
\[
\Theta_{n+2\ell, n'+2\ell} (R^{\UU_{n+2\ell}}_{\GG_\ell\times \UU_{n}}(\tau \otimes \pi))= R^{\UU_{n'+2\ell}}_{\GG_\ell\times \UU_{n'}}(\tau\otimes \pi').
\]
\end{proposition}

\subsection{See-saw dual pairs}
Recall the general formalism of see-saw dual pairs. Let $(G, G')$ and $(H, H')$ be two reductive dual pairs in a symplectic group $\rm{Sp}(W)$ such that $H \subset G$ and $G' \subset H'$.  Then there is a see-saw diagram
\[
\setlength{\unitlength}{0.8cm}
\begin{picture}(20,5)
\thicklines
\put(6.8,4){$G$}
\put(6.8,1){$H$}
\put(12.1,4){$H'$}
\put(12,1){$G'$}
\put(7,1.5){\line(0,1){2.1}}
\put(12.3,1.5){\line(0,1){2.1}}
\put(7.5,1.5){\line(2,1){4.2}}
\put(7.5,3.7){\line(2,-1){4.2}}
\end{picture}
\]
and the associated see-saw identity
\[
\langle \Theta_{G',G}(\pi_{G'}),\pi_H\rangle_H =\langle \pi_{G'},\Theta_{H,H'}(\pi_H)\rangle_{G'},
\]
where $\pi_H$ and $\pi_{G'}$ are representations of $H$ and $G'$ respectively.

In our case, if we put
\[
G=\UU_n\times \UU_n,\quad  G'=\UU_n\times \UU_1, \quad H=\UU_n,\ \mathrm{and}\ H'=\UU_{n+1},
\]
then the left-hand side of the see-saw identity concerns the basic case of  Fourier-Jacobi model whereas the right-hand side
concerns the basic case of Bessel model. In general, we need Proposition \ref{theta} and the following result which generalizes \cite[Proposition 5.2]{LW}.

\begin{proposition} \label{7.21}
Let $\pi$ be an irreducible unipotent representation of $\UU_n(\Fq)$, and $\pi'$ be an irreducible representation of $\UU_m(\Fq)$ with $n > m$ but $m \not\equiv n \ \mathrm{mod} \ 2$. Let $P$ be an $F$-stable maximal parabolic subgroup of $\UU_{n+1}$ with Levi factor $\GG_\ell \times \UU_m$ (so that $m + 2\ell = n + 1$). Let $\tau_1$ (resp. $\tau_2)$ be an irreducible cuspidal representations of $\GG_{\ell'}\fq$ (resp. $\GG_{\ell-\ell'}\fq$),  $\ell'\leq \ell$, which is nontrivial if $\ell'=1$ (resp. $\ell-\ell'=1$), and
\[
\tau=R_{\GG_{\ell'}\times  \GG_{\ell- \ell'}}^{\GG_\ell}(\tau_1\times\tau_2).
\]
Then we have
\[
m(\pi, \pi')=\langle \pi\otimes \bar{\nu}, \pi'\rangle _{H(\Fq)}=\langle R^{\UU_{n+1}}_{\GG_\ell \times \UU_m}(\tau\otimes\pi'),\pi\rangle _{\UU_n(\Fq)},
\]
where the data $(H,\nu)$ is given by \cite[(1.2)]{LW}.
\end{proposition}

\begin{proof}
It can be proved in the same way as \cite[Theorem 15.1]{GGP1}, where it was established for non-archimedean local fields, and the proof works for finite fields as well. We follow the notations in \cite{GGP1}. Let $V$ be an $n$-dimensional non-degenerate hermitian space and $W \subset V$ be an $m$-dimensional non-degenerate hermitian subspace, so that
\[
W^\bot=X+X^\vee+E.
\]
where $E = \bb{F}_{q^2} \cdot e$ is an anisotropic line and $X$ is an isotropic subspace with $\rm{dim} X = \ell-1$ and $X^\vee$ is the dual of $X$. Let
\[
E^- = \bb{F}_{q^2} \cdot f
\]
denote the rank 1 space equipped with a form which is the negative of that on $E$, so that $E + E^-$ is a split rank 2 space. The two isotropic lines in $E + E^-$  are spanned by
\[
v = e + f \textrm{ and } v' =\frac{1}{2\langle e,e\rangle} (e - f).
\]
Now consider the space
\[
W' = V \oplus E^-
\]
which contains $V$ with codimension 1 and isotropic subspaces
\[
Y = X +\bb{F}_{q^2}\cdot v \quad \textrm{and}\quad Y ^\vee= X^\vee +\bb{F}_{q^2}\cdot v'.
\]
Hence we have
\[
W' = Y + Y^\vee + W.
\]
Let $P = P (Y )$ be the parabolic subgroup of $U(W')$ stabilizing $Y $ and let $M(Y)$ be its Levi subgroup stabilizing $Y$ and $Y^\vee$. Then $\UU(V)=\UU_n$, $\UU(W')=\UU_{n+1}$ and $M(Y)=\GG_\ell \times \UU_m$. Let $P_V (X)$ be the parabolic subgroup of $\UU(V)$ stabilizing $X$, so that
\[
P_V(X) = M_V(X) \ltimes N_V(X)
\]
where $N_V(X)$ is the unipotent radical of $P_V(X)$. Let $Q$ be a subgroup of $P_V (X)$ and
\[
Q=(\GGL(X)\times \UU(W))\ltimes N_V(X).
\]
As in the proof \cite[Theorem 15.1]{GGP1}, one has
\[
\xymatrix
{
0  \ar[r] & N(Y)  \ar[r]  & P(Y)  \ar[r]  & \GGL(Y)\times \UU(W) \ar[r] & 0\\
0  \ar[r] & N(Y)\cap Q  \ar[r] \ar[u]  &Q  \ar[r] \ar[u] & R\times \UU(W) \ar[u] \ar[r] & 0
}
\]
where
$
R\subset \GGL(Y)
$
is the mirabolic subgroup which stabilizes the subspace $X \subset Y$ and fixes $v$ modulo $X$. Note also that $N(Y ) \cap Q \subset N_V(X)$ and
\[
N_V (X)/(N(Y) \cap Q)  \cong \rm{Hom}(E, X).
\]
As a consequence, one has
\[
(\tau\otimes\pi')|_Q=\tau|_R\otimes \pi'.
\]
By the proof of \cite[Theorem 15.1]{GGP1}, it suffices to show that
\[
\langle\pi,\rm{Ind}^{\UU(V)}_{Q}(\tau|_R\otimes \pi')\rangle_{\UU(V)}=\langle\pi,\rm{Ind}^{\UU(V)}_{Q}(\rm{Ind}^R_U\chi\otimes \pi')\rangle_{\UU(V)}.
\]

Let $N_n$ be the group of upper triangular unipotent matrices in $\GG_n\fq=\GGL_n(\bb{F}_{q^2})$. We fix a nontrivial character $\psi_0$ of $\bb{F}_{q^2}$ and let $\psi_n$ be the character of $N_n$, given by
\[
\psi_n(u) = \psi_0(u_{1,2} + u_{2,3} + \ldots + u_{n-1,n}).
\]
Let $R^n_i=\GGL_{i}(\bb{F}_{q^2})\times V_{n-i}$ be the subgroup of $\GGL_n(\bb{F}_{q^2})$ consisting of
\[
 \left(\begin{matrix}
   g      & v  \\
   0      & z
  \end{matrix}\right)
 \]
 with $g\in \GGL_{i}(\bb{F}_{q^2})$, $v\in M_{i\times n-i}$, $z\in N_{n-i}$.

By the theory of Bernstein-Zelevinsky derivatives (c.f. \cite[Corollary 4.3]{GGP2}),
\[
\tau|_R=\rm{Ind}^R_U\chi+\rm{Ind}^R_{R^\ell_{\ell'}}\tau_1\otimes\psi_{\ell-\ell'}+\rm{Ind}^R_{R^\ell_{\ell-\ell'}}\tau_2\otimes\psi_{\ell'}.
\]
Let $Q'$ be the subgroup of $Q$ given by
\[
Q'=(R^{\ell-1}_{\ell'}\times \UU(W))\ltimes (N(Y) \cap Q).
\]
Then there is an $F$-stable maximal parabolic subgroup $P_{\ell'}$ of $\UU_{n}$ with Levi factor $\GG_{\ell'} \times \UU_{n-2\ell'}$ such that $Q'\subset P_{\ell'}$. Thus we get
\[
\begin{aligned}
& \langle\pi,\rm{Ind}^{\UU(V)}_{Q}(\rm{Ind}^R_{R_{\ell'}}\tau_1\otimes\psi_{\ell-\ell'}\otimes \pi')\rangle_{\UU(V)}\\
=&\langle\pi,\rm{Ind}^{\UU(V)}_{Q'}(\tau_1\otimes\psi_{\ell-\ell'}\otimes \pi')\rangle_{\UU(V)}\\
=&\langle\pi,I^{\UU(V)}_{P_{\ell'}}(\tau_1\otimes\rm{Ind}^{\UU_{n-2\ell'}}_{\UU_{n-2\ell'}\cap Q'}(\psi_{\ell-\ell'}\otimes \pi'))\rangle_{\UU(V)}
\end{aligned}
\]
By our assumption, $\pi$ is unipotent and $\tau_1$ is not, hence
\[
\langle\pi,I^{\UU(V)}_{P_{\ell'}}(\tau_1\otimes\rm{Ind}^{\UU_{n-2\ell'}}_{\UU_{n-2\ell'}\cap Q'}(\psi_{\ell-\ell'}\otimes \pi'))\rangle_{\UU(V)}=0.
\]
In the same manner, one has
\[
\langle\pi,\rm{Ind}^{\UU(V)}_{Q}(\rm{Ind}^R_{R^\ell_{\ell-\ell'}}\tau_2\otimes\psi_{\ell'}\otimes \pi')\rangle_{\UU(V)}=0.
\]
It follows that
\[
\begin{aligned}
& \langle\pi,\rm{Ind}^{\UU(V)}_{Q}(\tau|_R\otimes \pi')\rangle_{\UU(V)}\\
=&\langle\pi,\rm{Ind}^{\UU(V)}_{Q}(\rm{Ind}^R_U\chi+\rm{Ind}^R_{R^\ell_{\ell'}}\tau_1\otimes\psi_{\ell-\ell'}+\rm{Ind}^R_{R^\ell_{\ell-\ell'}}\tau_2\otimes\psi_{\ell'})\otimes \pi'\rangle_{\UU(V)}\\
=&\langle\pi,\rm{Ind}^{\UU(V)}_{Q}(\rm{Ind}^R_U\chi\otimes \pi')\rangle_{\UU(V)},
\end{aligned}
\]
which completes the proof.
\end{proof}

Similarly, in the Fourier-Jacobi case we have the following result, which generalizes \cite[Proposition 6.5]{LW}.

\begin{proposition}\label{7.31}
Let $\pi$ be an irreducible unipotent representation of $\UU_n(\Fq)$, and $\pi'$ be an irreducible representation of $\UU_m(\Fq)$ with $n > m$ and $m \equiv n \ \mathrm{mod}  \ 2$. Let $P$ be an $F$-stable maximal parabolic subgroup of $\UU_m$ with Levi factor $\GG_\ell \times \UU_m$ (so that $m + 2\ell = n $). Let $\tau_1$ (resp. $\tau_2)$ be an irreducible cuspidal representations of $\GG_{\ell'}\fq)$ (resp. $\GG_{\ell-\ell'}\fq$), $\ell'\leq\ell$, which is nontrivial if $\ell'=1$ (resp. $\ell-\ell'=1$), and
\[
\tau=R_{\GG_{\ell'}\times  \GG_{\ell- \ell'}}^{\GG_\ell}(\tau_1\times\tau_2).
\]
Then we have
\[
m(\pi, \pi')=\langle \pi\otimes\bar{\nu},\pi'\rangle_{H(\Fq)}=\langle  \pi\otimes \omega_n, R_{\GG_\ell \times \UU_m}^{\UU_{n}}(\tau\otimes\pi')\rangle _{\UU_{n}(\Fq)},
\]
where the data $(H,\nu)$ is given by \cite[(1.6)]{LW}.
\end{proposition}

% We also recall from \cite[Proposition 6.4]{LW} that a unipotent representation is self-dual.

In summary, to determine $m(\pi, \pi')$ it suffices to calculate   $\langle R^{\UU_{n+1}}_{\GG_\ell \times \UU_m}(\tau\otimes\pi'),\pi\rangle _{\UU_n(\Fq)}$ or $\langle  \pi\otimes \omega_n, R_{\GG_\ell \times \UU_m}^{\UU_{n}}(\tau\otimes\pi')\rangle _{\UU_{n}(\Fq)}$, which will be done by see-saw arguments and induction on $n$.

\subsection{Theta correspondence of unipotent representations}\label{ssec3.2}
Let us recall the theta correspondence between unipotent representations of finite unitary groups. 
We say that two partitions $\mu=[\mu_i]$ and $\mu'=[\mu_i']$ are {\sl close} if  $|\mu_i-\mu_i'|\le1$ for every $i$, and that $\mu$ is {\sl even} if $\#\{i|\mu_i=j\}$ is even for any $j>0$, i.e. every part of $\mu$ occurs with even multiplicities. Let
\[
\mu\cap\mu'=[\mu_i]_{\{ i| \mu_i=\mu_i'\}}
\]
be the partition formed by the common parts of $\mu$ and $\mu'$. Following \cite{AMR}, we say that $\mu$ and $\mu'$ are {\sl 2-transverse} if they are close and $\mu\cap \mu'$ is even.
In particular, if $\mu$ and $\mu'$ are close and $\mu\cap \mu'=\emptyset$, then $\mu$ and $\mu'$ are 2-transverse, and in this case we say that they are {\sl transverse}. For example, let
$\lambda=[\lambda_1,\ldots, \lambda_k]$ be a partition of $n$, and let
\[
\lambda_*=[\lambda_2,\ldots, \lambda_k]
\]
be the partition of $n-\lambda_1$ obtained by removing the first row of $\lambda$. Then ${}^t\lambda$ and ${}^t\lambda_*$ are transverse. Moreover, $\lambda_*$ is the unique partition of $n-\lambda_1$ such that ${}^t\lambda$ and ${}^t\lambda_*$ are 2-transverse.

For  partitions $\lambda$ and $\lambda'$ of $n$ and $n'$ respectively, denote the multiplicity of $\pi_{\lambda}\otimes\pi_{\lambda'}$ in $\omega_{n,n'}$ by $m_{\lambda,\lambda'}$. By \cite{AMR} Theorem 4.3, Lemma 5.3 and Lemma 5.4, we have

\begin{proposition}\label{u1}
With above notations,
\[
m_{\lambda,\lambda'}=\left\{
\begin{array}{ll}
1, & \textrm{if  }{}^t\lambda\textrm{ and }{}^t\lambda'\textrm{ are }2\textrm{-transverse},\\
0, & \textrm{otherwise.}
\end{array}
\right.
\]
In other words,
\[
\Theta_{n,n'}(\pi_{\lambda})=\bigoplus_{\mbox{\tiny$\begin{array}{c}{}^t\lambda\ \mathrm{and}\ {}^t\lambda'\mathrm{\ are\ }2\textrm{-}\mathrm{transverse}\\
|\lambda'|=n' \end{array}$}}\pi_{\lambda'}
\]
\end{proposition}

\begin{corollary}\label{CU1}
Let $\lambda=[\lambda_1,\lambda_2,\ldots,\lambda_k]$ be a partition of $n$. Then the following hold.

(i) If $n'<n-\lambda_1$, then $\Theta_{n,n'}(\pi_{\lambda})=0$.

(ii) If $n'=n-\lambda_1$, then $\Theta_{n,n'}(\pi_{\lambda})=\pi_{\lambda_*}$ with $\lambda_*=[\lambda_2,\ldots,\lambda_k]$.
\end{corollary}
\begin{proof}
If $\pi_\mu\in \Theta_{n,n'}(\pi_\lambda)$, then by Proposition \ref{u1}, $^t\mu$ and $^t\lambda$ are close, which implies that
\[
^t\mu_{i}\ge {^t\lambda_i-1}\quad \mathrm{for}\quad i=1,2\ldots,\lambda_1.
\]
It follows that
\[
|\mu|\ge\sum_{i=1}^{\lambda_1}{^t\mu_{i}}\ge\sum_{i=1}^{\lambda_1}(^t\lambda_i-1)=n-\lambda_1,
\]
and therefore
\[
\Theta_{n,n'}(\pi_\lambda)=\left\{
\begin{array}{ll}
0, &  \textrm{if } n'<n-\lambda_1,\\
\pi_{\lambda_*}, &  \textrm{if } n'=n-\lambda_1.
\end{array}\right.
\]
\end{proof}

\begin{corollary}\label{CU2}
With above notations, if $n'\ge n+\lambda_1-1 $ and $\pi_\mu\subset\Theta_{n,n'}(\pi_\lambda)$, then $\mu_1\ge \lambda_1$.
\end{corollary}
\begin{proof}
By Proposition \ref{u1}, $^t\mu$ and $^t\lambda$ are close, hence
\[
^t\mu_{i}\le {^t\lambda_i+1}\quad \mathrm{for}\quad i=1,2\ldots,\lambda_1.
\]
It follows that
\[
\sum_{i=1}^{\lambda_1-1}{^t\mu_{i}}\le\sum_{i=1}^{\lambda_1-1}(^t\lambda_i+1)=n-{^t\lambda_{\lambda_1}}+\lambda_1-1 \le n+\lambda_1-2< n'.
\]
Therefore $^t\mu_{\lambda_1}>0$, i.e. $\mu_1\ge \lambda_1$.
\end{proof}

A 2-\emph{hook} of a partition $\lambda$ is a pair of blocks of the form $\{(i, j), (i, j + 1)\}$ or $\{(i, j), (i + 1, j)\}$ on the boundary of $\lambda$, such that we still obtain a Young diagram by removing these blocks  from $\lambda$. A 2-hook of the above forms is called of type $(1^2)$ or (2) respectively. If  $\mu$ is obtained from $\lambda$ by removing a 2-hook, then we also say that $\lambda$ is obtained from $\mu$ by adding a 2-hook.

\begin{corollary} \label{CU3}
Keep notations as above. Then the following hold.

(i) If $\pi_\mu\subset\Theta_{n,n+m}(\pi_\lambda)$, $m\ge \lambda_1$ and $\mu_1\le m+2$, then $\mu$ is obtained from $[m+2,\lambda]$ by removing a $2$-hook;

(ii) If $\pi_\mu\subset\Theta_{n,n-\lambda_1+2}(\pi_\lambda)$, then $\mu$ is obtained from $\lambda_*=[\lambda_2,\lambda_3,\cdots,\lambda_k]$ by adding a $2$-hook.
\end{corollary}

\begin{proof}
We will only prove (i), and the proof of (ii) is similar.

Since $\pi_\mu\subset \Theta_{n,n+m}(\pi_\lambda)$, $^t\lambda_i-1\le{}^t\mu_i\le {}^t\lambda_i+1$, which implies that
\[
n+m=|\mu|=\sum_{i} {}^t\mu_i\le\sum_{i=1}^{m+2} {}^t\lambda_i+1=n+m+2.
\]
If there exists $j\in [1,m+2]$ such that $^t\mu_j={}^t\lambda_j-1$, then
\[
n+m=|\mu|={}^t\mu_j+\sum_{i\ne j} {}^t\mu_i\le{}^t\lambda_j-1+\sum_{i=1,i\ne j}^{m+2} {}^t\lambda_i+1=n+m.
\]
It follows that in this case $^t\mu_i={}^t\lambda_i+1$ if $i\ne j$. In other words, $\mu$ is obtained by removing two blocks from the $j$-th column of $[m+2,\lambda]$. Since $\mu$ is a partition,  these two blocks form a  2-hook of $[m+2,\lambda]$.

Next suppose that $^t\mu_i\ge{}^t\lambda_i$ for $i=1,\cdots,m+2$. It is easy to see that there exist $j<j'$ such that $^t\mu_j={}^t\lambda_j$ and $^t\mu_{j'}={}^t\lambda_{j'}$. Since $\mu$ and $\lambda$ are even, we must have $j'=j+1$, which implies that $\mu$ is obtained by removing two blocks from  the $^t\lambda_j$-th row of $[m+2,\lambda]$.
\end{proof}

\section{The Gan-Gross-Prasad problem}\label{sec4}

In Section \ref{sec3}, Proposition \ref{7.21} and Proposition \ref{7.31} show that parabolic induction preserves multiplicities, which are finite field analogs of Theorem 15.1 and Theorem 16.1 in \cite{GGP1} respectively for unipotent representations. This reduces the calculation to the basic case. In this section we prove Theorem \ref{1.1} using the theta correspondence and see-saw dual pairs.

\begin{lemma}\label{0}
Let $\lambda=[\lambda_1,\ldots, \lambda_k]$ and $\lambda'=[\lambda_1',\ldots, \lambda'_{k'}]$ be partitions of $n$ and $m$ respectively, $n>m$. If $\lambda'_1>\lambda_1+1$ or $\lambda'_1<\lambda_1-1$, then
\[
m(\pi_\lambda,\pi_{\lambda'})=0.
\]
\end{lemma}

\begin{proof}
We will only prove the Bessel case. The proof for the Fourier-Jacobi case is similar and will be omitted. By \cite[Proposition  5.2]{LW}, we only need to compute
\[
\langle \pi_\lambda,R^{\UU_{n+1}}_{\GG_\ell \times \UU_m}(\tau\otimes\pi_{\lambda'})\rangle_{\UU_n(\Fq)},
\]
where $\tau$ is an irreducible cuspidal representation of $\GG_\ell\fq$. We also assume that $\tau$ is not conjugate self-dual.

(i) Suppose that $\lambda_1'>\lambda_1+1$. Then by Corollary \ref{CU1} (i),
\begin{equation}\label{eq41}
\Theta_{n,n+1-\lambda_1'}(\pi_\lambda)=0.
 \end{equation}
Put $\lambda'_*=[\lambda_2',\ldots, \lambda_{k'}']$. By Corollary \ref{CU1} (ii),
 \[
 \pi_{\lambda'}\subset \Theta_{m-\lambda_1',m}(\pi_{\lambda'_*}).
 \]
By Proposition \ref{theta}, we have
 \[
 R^{\UU_{n+1}}_{\GG_\ell \times \UU_m}(\tau\otimes\pi_{\lambda'}) \subset \Theta_{n+1-\lambda_1',n+1}(R^{\UU_{n+1-\lambda'_1}}_{\GG_\ell \times \UU_{m-\lambda_1'}}(\tau\otimes\pi_{\lambda'_*})).
\]

Consider the see-saw diagram
\[
\setlength{\unitlength}{0.8cm}
\begin{picture}(20,5)
\thicklines
\put(5.6,4){$\UU_{n+1-\lambda_1'}\times \UU_{n+1-\lambda_1'}$}
\put(6.7,1){$\UU_{n+1-\lambda_1'}$}
\put(12.7,4){$\UU_{n+1}$}
\put(12.3,1){$\UU_{n}\times \UU_1$}
\put(7.7,1.5){\line(0,1){2.1}}
\put(12.8,1.5){\line(0,1){2.1}}
\put(8,1.5){\line(2,1){4.2}}
\put(8,3.7){\line(2,-1){4.2}}
\end{picture}
\]
By the see-saw identity, Proposition \ref{theta} and (\ref{eq41}), one has
\[
\begin{aligned}
&\langle \pi_\lambda,R^{\UU_{n+1}}_{\GG_\ell \times \UU_m}(\tau\otimes\pi_{\lambda'})\rangle_{\UU_n(\Fq)}\\
\le& \langle \pi_\lambda,\Theta_{n+1-\lambda_1',n+1}(R^{\UU_{n+1-\lambda_1'}}_{\GG_\ell \times \UU_{m-\lambda_1'}}(\tau\otimes\pi_{\lambda'_*}))\rangle_{\UU_n(\Fq)}\\
=& \langle \Theta_{n,n+1-\lambda_1'}(\pi_\lambda)\otimes \omega_{n+1-\lambda_1'},R^{\UU_{n+1-\lambda'_1}}_{\GG_\ell \times \UU_{m-\lambda_1'}}(\tau\otimes\pi_{\lambda'_*})\rangle_{\UU_{n+1-\lambda_1'}(\Fq)}\\
=& 0.
\end{aligned}
\]

(ii) Suppose that $\lambda_1'<\lambda_1-1$. Put $\mu=[\lambda_1-2,\lambda_1',\lambda_2',\ldots, \lambda_{k'}']$. Then by Corollary \ref{CU1} (ii) and Proposition \ref{theta}, one has
\[
\Theta_{n+\lambda_1-1,n+1}(R^{\UU_{n+\lambda_1-1}}_{\GG_\ell \times \UU_{m+\lambda_1-2}}(\tau\otimes\pi_{\mu}))=R^{\UU_{n+1}}_{\GG_\ell \times \UU_{m}}(\tau\otimes\pi_{\lambda'}).
 \]

 Consider the see-saw diagram
\[
\setlength{\unitlength}{0.8cm}
\begin{picture}(20,5)
\thicklines
\put(5.8,4){$\UU_{n+\lambda_1-1}\times \UU_{n+\lambda_1-1}$}
\put(6.9,1){$\UU_{n+\lambda_1-1}$}
\put(12.7,4){$\UU_{n+1}$}
\put(12.3,1){$\UU_{n}\times \UU_1$}
\put(7.7,1.5){\line(0,1){2.1}}
\put(12.8,1.5){\line(0,1){2.1}}
\put(8,1.5){\line(2,1){4.2}}
\put(8,3.7){\line(2,-1){4.2}}
\end{picture}
\]
Similar to the proof in (i), one has
\begin{equation}\label{6.A}
\begin{aligned}
&
\langle \pi_\lambda,R^{\UU_{n+1}}_{\GG_\ell \times \UU_{m}}(\tau\otimes\pi_{\lambda'})\rangle_{\UU_n(\Fq)}\\
=& \langle \pi_\lambda,\Theta_{n+\lambda_1-1,n+1}(R^{\UU_{n+\lambda_1-1}}_{\GG_\ell \times \UU_{m+\lambda_1-2}}(\tau\otimes\pi_{\mu}))\rangle_{\UU_n(\Fq)}\\
=& \langle \Theta_{n,n+\lambda_1-1}(\pi_\lambda)\otimes \omega_{n+\lambda_1-1},R^{\UU_{n+\lambda_1-1}}_{\GG_\ell \times \UU_{m+\lambda_1-2}}(\tau\otimes\pi_{\mu})\rangle_{\UU_{n+\lambda_1-1}(\Fq)}.
\end{aligned}
\end{equation}
By Proposition \ref{u1},
\[
\Theta_{n,n+\lambda_1-1}(\pi_{\lambda})=\bigoplus_{\mbox{\tiny$\begin{array}{c}{}^t\lambda\ \mathrm{and}\ {}^t\mu'\mathrm{\ are\ }2\textrm{-}\mathrm{transverse}\\
|\mu'|=n+\lambda_1-1 \end{array}$}}\pi_{\mu'}.
\]
It suffices to prove that
\[
\langle \pi_{\mu'}\otimes \omega_{n+\lambda_1-1},R^{\UU_{n+\lambda_1-1}}_{\GG_\ell \times \UU_{m+\lambda_1-2}}(\tau\otimes\pi_{\mu})\rangle_{\UU_{n+\lambda_1-1}(\Fq)}=0
\]
for any partition $\mu'$ in the above direct sum.

By Corollary \ref{CU2}, if $^t\lambda\ \mathrm{and}\ {}^t\mu'\mathrm{\ are\
}2\textrm{-}\mathrm{transverse}$ and $|\mu'|=n+\lambda_1-1 $, then $\mu'_1\ge \lambda_1$.
Fix $\mu'=[\mu_1',\mu'_2,\ldots,\mu_l']$, and put $\mu'_*=[\mu'_2,\ldots,\mu_l']$. Then $|\mu'_*|=n+\lambda_1-1-\mu_1'<n$. By Corollary \ref{CU1} (ii), one has
 \[
 \pi_{\mu'} \subset \Theta_{n+\lambda_1-1-\mu_1',n+\lambda_1-1}(\pi_{\mu'_*}).
 \]
 Consider the see-saw diagram
\[
\setlength{\unitlength}{0.8cm}
\begin{picture}(20,5)
\thicklines
\put(5.8,4){$\UU_{n+\lambda_1-1}\times \UU_{n+\lambda_1-1}$}
\put(6.9,1){$\UU_{n+\lambda_1-1}$}
\put(12.0,4){$\UU_{n+\lambda_1-\mu_1'}$}
\put(11.3,1){$\UU_{n+\lambda_1-1-\mu_1'}\times \UU_1$}
\put(7.7,1.5){\line(0,1){2.1}}
\put(12.8,1.5){\line(0,1){2.1}}
\put(8,1.5){\line(2,1){4.2}}
\put(8,3.7){\line(2,-1){4.2}}
\end{picture}
\]
Then
\begin{equation}\label{6.B}
\begin{aligned}
&\langle \pi_{\mu'}\otimes \omega_{n+\lambda_1-1},R^{\UU_{n+\lambda_1-1}}_{\GG_\ell \times \UU_{m+\lambda_1-2}}(\tau\otimes\pi_{\mu})\rangle_{\UU_{n+\lambda_1-1}(\Fq)}\\
\le &\langle \Theta_{n+\lambda_1-1-\mu_1',n+\lambda_1-1}(\pi_{\mu'_*})\otimes \omega_{n+\lambda_1-1},R^{\UU_{n+\lambda_1-1}}_{\GG_\ell \times \UU_{m+\lambda_1-2}}(\tau\otimes\pi_{\mu})\rangle_{\UU_{n+\lambda_1-1}(\Fq)}\\
= &\langle \pi_{\mu'_*},\Theta_{n+\lambda_1-1,n+\lambda_1-\mu_1'}(R^{\UU_{n+\lambda_1-1}}_{\GG_\ell \times \UU_{m+\lambda_1-2}}(\tau\otimes\pi_{\mu}))\rangle_{\UU_{n+\lambda_1-1-\mu_1'}(\Fq)}.\\
\end{aligned}
\end{equation}
Since $\mu_1'-1\ge\lambda_1-1>\lambda_1-2=\mu_1$, by Corollary \ref{CU1} (i) and Proposition \ref{theta}, \[
\Theta_{n+\lambda_1-1,n+\lambda_1-\mu_1'}(R^{\UU_{n+\lambda_1-1}}_{\GG_\ell \times \UU_{m+\lambda_1-2}}(\tau\otimes\pi_{\mu}))=0.
\]
Hence by (\ref{6.A}) and (\ref{6.B}), we have
\[
\begin{aligned}
\langle \pi_\lambda,R^{\UU_{n+1}}_{\GG_\ell \times \UU_{m}}(\tau\otimes\pi_{\lambda'})\rangle_{\UU_n(\Fq)}&=\langle \Theta_{n,n+\lambda_1-1}(\pi_\lambda)\otimes \omega_{n+\lambda_1-1},R^{\UU_{n+\lambda_1-1}}_{\GG_\ell \times \UU_{m+\lambda_1-2}}(\tau\otimes\pi_{\mu})\rangle_{\UU_{n+\lambda_1-1}(\Fq)}=0.
\end{aligned}
\]

\end{proof}

\begin{theorem}\label{UU}
Let $\lambda$ and $\lambda'$ be partitions of $n$ and $m$ respectively, $n\geq m$. Then
\[
m(\pi_\lambda, \pi_{\lambda'})=\left\{
\begin{array}{ll}
1, &  \textrm{if }\
\lambda \textrm{ and }
\lambda'
 \textrm{ are }  2\textrm{-transverse},\\
0, & \textrm{otherwise.}
\end{array}\right.
\]
\end{theorem}

\begin{proof}
We prove the proposition by induction on $n$. If $n=1$, then the Bessel case is vacuum and the Fourier-Jacobi case follows from
\[
\langle{\bf 1}\otimes\omega_{1},{\bf 1}\rangle_{\UU_1(\Fq)}=0.
\]

Suppose that the proposition holds for $n'<n$. Then we will prove the Bessel case for $n$. The proof for the Fourier-Jacobi case is similar and will be left to the reader.

As before, write $\lambda=[\lambda_1,\ldots,\lambda_k]$ which is a partition of $n$ into $k$ rows and $l:=\lambda_1$ columns. Similarly write $\lambda'=[\lambda_1',\ldots, \lambda'_{k'}]$ and put $l':=\lambda_1'$.  By Lemma \ref{0}, there are only three cases for $\lambda_1'$ to be considered: $\lambda_1'=\lambda_1-1$, $\lambda_1$ or $\lambda_1+1$.

(i) Suppose that $\lambda_1'=\lambda_1+1$. Put $\lambda_*=[\lambda_2,\ldots, \lambda_{k}]$ and $\lambda'_*=[\lambda_2',\ldots, \lambda_{k'}']$. By Corollary \ref{CU1} (ii), we have
 \[
 \Theta_{n,n+1-\lambda_1'}(\pi_{\lambda})=\Theta_{n,n-\lambda_1}(\pi_{\lambda}) = \pi_{\lambda_*}.
 \]
Let $\tau$ be an irreducible cuspidal representation of $\GG_\ell\fq$ which is not conjugate self-dual. By Corollary \ref{CU3} (i) and Proposition \ref{theta},
  \begin{equation}\label{ggp1}
 \Theta_{n+1-\lambda_1',n+1}(R^{\UU_{n+1-\lambda_1'}}_{\GG_\ell \times \UU_{m-\lambda_1'}}(\tau\otimes\pi_{\lambda'_*}))=R^{\UU_{n+1}}_{\GG_\ell \times \UU_{m}}(\tau\otimes\pi_{\lambda'})+\sum_{\mbox{\tiny$\begin{array}{c}{}\mu\in \Theta_{m-\lambda_1',m}(\lambda'_*)\\
\mu_1>\lambda_1'+1\end{array}$}}R^{\UU_{n+1}}_{\GG_\ell \times \UU_{m}}(\tau\otimes\pi_{\mu'}).
 \end{equation}
 By Lemma \ref{0}, for any $\mu$ with $\mu_1>\lambda_1'+1>\lambda_1+1$,
  \begin{equation}\label{ggp2}
 \langle \pi_\lambda,R^{\UU_{n+1}}_{\GG_\ell \times \UU_{m}}(\tau\otimes\pi_{\mu'})\rangle_{\UU_n(\Fq)}=0.
 \end{equation}

 Consider the see-saw diagram
\[
\setlength{\unitlength}{0.8cm}
\begin{picture}(20,5)
\thicklines
\put(5.5,4){$\UU_{n+1-\lambda_1'}\times \UU_{n+1-\lambda_1'}$}
\put(6.6,1){$\UU_{n+1-\lambda_1'}$}
\put(12.3,4){$\UU_{n+1}$}
\put(12.0,1){$\UU_{n}\times \UU_1$}
\put(7.7,1.5){\line(0,1){2.1}}
\put(12.8,1.5){\line(0,1){2.1}}
\put(8,1.5){\line(2,1){4.2}}
\put(8,3.7){\line(2,-1){4.2}}
\end{picture}
\]
Similar to the proof of Lemma \ref{0}, by (\ref{ggp1}) and (\ref{ggp2}) one has
\[
\begin{aligned}
&\langle \pi_\lambda,R^{\UU_{n+1}}_{\GG_\ell \times \UU_{m}}(\tau\otimes\pi_{\lambda'})\rangle_{\UU_n(\Fq)}\\
=&\langle \pi_\lambda,\Theta_{n+1-\lambda_1',n+1}(R^{\UU_{n+1-\lambda'_1}}_{\GG_\ell \times \UU_{m-\lambda_1'}}(\tau\otimes\pi_{\lambda'_*}))\rangle_{\UU_n(\Fq)}\\
=&\langle \Theta_{n,n+1-\lambda_1'}(\pi_\lambda)\otimes \omega_{n+1-\lambda_1'},R^{\UU_{n+1-\lambda_1'}}_{\GG_\ell \times \UU_{m-\lambda_1'}}(\tau\otimes\pi_{\lambda'_*})\rangle_{\UU_{n+1-\lambda_1'}(\Fq)}\\
=&\langle \pi_{\lambda_*}\otimes \omega_{n+1-\lambda_1'},R^{\UU_{n+1-\lambda_1'}}_{\GG_\ell \times \UU_{m-\lambda_1'}}(\tau\otimes\pi_{\lambda'_*})\rangle_{\UU_{n+1-\lambda_1'}(\Fq)}\\
=& m(\pi_{\lambda_*}, \pi_{\lambda'_*}).
\end{aligned}
\]
By induction hypothesis on $n$ for the Fourier-Jacobi case, one has
\[
m(\pi_{\lambda_*},\pi_{\lambda'_*})=\left\{
\begin{array}{ll}
1, &  \textrm{if }\ \lambda_* \textrm{ and } \lambda'_* \textrm{ are }  2\textrm{-transverse},\\
0, & \textrm{otherwise.}
\end{array}\right.
\]
Since $\lambda_1\ne \lambda_1'$, it is clear that $\lambda$ and $\lambda'$ are 2-transverse if and only if $\lambda_*$ and $\lambda'_*$ are 2-transverse.

(ii) Suppose that $\lambda_1'=\lambda_1$. Let $\lambda'_*$ and $\tau$ be as above. Similar to the proof of (i), one has
\[
\begin{aligned}
& \langle \pi_\lambda,R^{\UU_{n+1}}_{\GG_\ell \times \UU_{m}}(\tau\otimes\pi_{\lambda'})\rangle_{\UU_n(\Fq)}\\
=&\langle \pi_\lambda,\Theta_{n+1-\lambda_1',n+1}(R^{\UU_{n+1-\lambda'_1}}_{\GG_\ell \times \UU_{m-\lambda_1'}}(\tau\otimes\pi_{\lambda'_*}))\rangle_{\UU_n(\Fq)}\\
=&\langle \Theta_{n,n+1-\lambda_1'}(\pi_\lambda)\otimes \omega_{n+1-\lambda_1'},R^{\UU_{n+1-\lambda_1'}}_{\GG_\ell \times \UU_{m-\lambda_1'}}(\tau\otimes\pi_{\lambda'_*})\rangle_{\UU_{n+1-\lambda_1'}(\Fq)}.\\
\end{aligned}
\]

First assume that $\lambda_1>\lambda_2$. Then we claim that $\Theta_{n,n+1-\lambda_1'}(\pi_\lambda)=0$. In fact if $\pi_{\mu'}\subset \Theta_{n,n+1-\lambda_1'}(\pi_\lambda)$, then by Proposition \ref{u1} one has
\[
^t\mu'_{i}\ge {}^t\lambda_i-1\quad \mathrm{for}\quad i=1,2\ldots,l.
\]
It follows that
\begin{equation}\label{ggp3}
n+1-\lambda_1=|\mu'|\geq \sum_{i=1}^l{^t\mu_{i}'}\ge\sum_{i=1}^{l}(^t\lambda_i-1)=n-\lambda_1.
\end{equation}
We have three cases for $\mu'_1$:
\begin{itemize}
\item
If $\mu'_1> l+1$, then
\[
n+1-\lambda_1=|\mu'|\ge\sum_{i=1}^{l+2}{^t\mu_{i}'}\ge 2+\sum_{i=1}^l{^t\mu_{i}'}\geq n+2-\lambda_1
\]
which is impossible.
\item If $\mu_1'= l+1$, then by (\ref{ggp3}),
\[
^t\mu_i'=\left\{
\begin{array}{ll}
^t\lambda_i-1, &  \textrm{if }\ i\le l ,\\
1, & \textrm{if }\ i=l+1.
\end{array}\right.
\]
Since $\lambda_1> \lambda_2$, we have $^t\lambda_l=1$ and thus $^t\mu_l'={}^t\lambda_l-1=0<{^t\mu_{l+1}}$, which is impossible.

\item It follows that $\mu_1'\le l=\lambda_1$. By (\ref{ggp3}), there exist $j\in [1,l]$ such that
\[
^t\mu_i'=\left\{
\begin{array}{ll}
^t\lambda_i-1, &  \textrm{if }\ i\ne j \textrm{ and }i\in [1,l],\\
^t\lambda_i, &\textrm{if }\  i=j.
\end{array}\right.
\]
Therefore $^t\mu'\cap{}^t\lambda=[^t\lambda_j]$ is not even. This contradicts the fact that $\pi_\mu'\subset \Theta_{n,n+1-\lambda_1'}(\pi_\lambda)$.
\end{itemize}
This proves the claim that $\Theta_{n,n+1-\lambda_1'}(\pi_\lambda)=0$.
Hence if $\lambda_1\ne\lambda_2$, then
\[
\begin{aligned}
\langle \pi_\lambda,R^{\UU_{n+1}}_{\GG_\ell \times \UU_{m}}(\tau\otimes\pi_{\lambda'})\rangle_{\UU_n(\Fq)}&=\langle \Theta_{n,n+1-\lambda_1'}(\pi_\lambda)\otimes \omega_{n+1-\lambda_1'},R^{\UU_{n+1-\lambda_1'}}_{\GG_\ell \times \UU_{m-\lambda_1'}}(\tau\otimes\pi_{\lambda'_*})\rangle_{\UU_{n+1-\lambda_1'}(\Fq)}=0.
\end{aligned}
\]
Moreover, by our assumption $\lambda_1=\lambda_1'$, hence $\#\{i|\lambda_i=\lambda_i'=\lambda_1\}=1$, which shows that $\lambda$ and $\lambda'$ are not 2-transverse.

Next assume that $\lambda_1=\lambda_2$. By the above discussion, $\pi_\mu'\subset \Theta_{n,n+1-\lambda_1'}(\pi_\lambda)$ if and only if
\[
^t\mu_i'=\left\{
\begin{array}{ll}
^t\lambda_i-1, &  \textrm{if } i\le l ,\\
1, &\textrm{if }i=l+1.
\end{array}\right.
\]
In other words, $\mu'=[\lambda_2+1,\lambda_3,\lambda_4,\ldots,\lambda_k]$. Then by induction on $n$,
\[
\begin{aligned}
&\langle \pi_\lambda,R^{\UU_{n+1}}_{\GG_\ell \times \UU_{m}}(\tau\otimes\pi_{\lambda'})\rangle_{\UU_n(\Fq)}\\
=&\langle \Theta_{n,n+1-\lambda_1'}(\pi_\lambda)\otimes \omega_{n+1-\lambda_1'},R^{\UU_{n+1-\lambda_1'}}_{\GG_\ell \times \UU_{m-\lambda_1'}}(\tau\otimes\pi_{\lambda'_*})\rangle_{\UU_{n+1-\lambda_1'}(\Fq)}\\
=&\langle \pi_{\mu'}\otimes \omega_{n+1-\lambda_1'},R^{\UU_{n+1-\lambda_1'}}_{\GG_\ell \times \UU_{m-\lambda_1'}}(\tau\otimes\pi_{\lambda'_*})\rangle_{\UU_{n+1-\lambda_1'}(\Fq)}
\\
=& m(\pi_{\mu'}, \pi_{\lambda'_*})\\
=&\left\{
\begin{array}{ll}
1, &  \textrm{if }\ \mu' \textrm{ and } \lambda'_* \textrm{ are }  2\textrm{-transverse},\\
0, & \textrm{otherwise.}
\end{array}\right.
\end{aligned}
\]
If $\mu'$ and $\lambda'_*$ are not 2-transverse, then it is clear that $\lambda$ and $\lambda'$ are not 2-transverse. On the other hand, if $\mu'$ and $\lambda'_*$ are 2-transverse, then $|\lambda_2+1-\lambda_2'|\le 1$, which implies that $\lambda_2'\ge \lambda_2=\lambda_1$. On the other hand $\lambda_2'\le\lambda_1'=\lambda_1$, which implies that $\lambda_2'=\lambda_1=\lambda_1'=\lambda_2$. Hence if $\mu'$ and $\lambda'_*$ are 2-transverse, so are $\lambda$ and $\lambda'$.

(iii) Suppose that $\lambda_1'=\lambda_1-1$. Let $\lambda_*$ and $\lambda'_*$ be as above, and
\[
\tau= R^{\GG_\ell}_{{\GG_1}\times {\GG_{\ell-1}}}(\tau_1\otimes\tau_2),
\]
where $\tau_1$ and $\tau_2$ are irreducible  cuspidal representations  of ${\GGL_1(\bb{F}_{q^2})}$ and ${\GG_{\ell-1}}(\bb{F}_{q^2})$ respectively that are not conjugate self-dual. By Proposition \ref{7.21}, we only need to compute
\[
\langle \pi_\lambda,R^{\UU_{n+1}}_{\GG_\ell \times \UU_{m}}(\tau\otimes\pi_{\lambda'})\rangle_{\UU_n(\Fq)}.
\]
We have
\[
\begin{aligned}
& \langle \pi_\lambda,R^{\UU_{n+1}}_{\GG_\ell \times \UU_{m}}(\tau\otimes\pi_{\lambda'})\rangle_{\UU_n(\Fq)}\\
=&\langle \pi_\lambda,R^{\UU_{n+1}}_{\GG_\ell \times \UU_{m}}(R^{\GG_\ell}_{{\GG_1}\times {\GG_{\ell-1}}}(\tau_1\otimes\tau_2)\otimes\pi_{\lambda'})\rangle_{\UU_n(\Fq)}\\
=&\langle \pi_\lambda,R^{\UU_{n+1}}_{\GG_\ell \times \UU_{m}}(\tau_1\otimes
R^{\UU_{n-1}}_{\GG_{\ell-1} \times \UU_{m}}(\tau_2\otimes\pi_{\lambda'}))\rangle_{\UU_n(\Fq)}\\
=& m(\pi_\lambda, R^{\UU_{n-1}}_{\GG_{\ell-1} \times \UU_{m}}(\tau_2\otimes\pi_{\lambda'}))\\
=&\langle \pi_\lambda,
R^{\UU_{n-1}}_{\GG_{\ell-1} \times \UU_{m}}(\tau_2\otimes\pi_{\lambda'})\rangle_{\UU_{n-1}(\Fq)}.
\end{aligned}
\]
Consider the see-saw diagram
\[
\setlength{\unitlength}{0.8cm}
\begin{picture}(20,5)
\thicklines
\put(6.2,4){$\UU_{n-\lambda_1}\times \UU_{n-\lambda_1}$}
\put(7,1){$\UU_{n-\lambda_1}$}
\put(12.4,4){$\UU_{n}$}
\put(11.4,1){$\UU_{n-1}\times \UU_1$}
\put(7.7,1.5){\line(0,1){2.1}}
\put(12.8,1.5){\line(0,1){2.1}}
\put(8,1.5){\line(2,1){4.2}}
\put(8,3.7){\line(2,-1){4.2}}
\end{picture}
\]
Similar to  case (i), applying Corollary \ref{CU3} and Lemma \ref{0}, one can show that
 \[
\begin{aligned}
& \langle \pi_\lambda,
R^{\UU_{n-1}}_{\GG_{\ell-1} \times \UU_{m}}(\tau_2\otimes\pi_{\lambda'})\rangle_{\UU_{n-1}(\Fq)}\\
=&\langle \Theta_{n-\lambda_1,n}(\pi_{\lambda_*}),
R^{\UU_{n-1}}_{\GG_{\ell-1} \times \UU_{m}}(\tau_2\otimes\pi_{\lambda'})\rangle_{\UU_{n-1}(\Fq)}\\
=&\langle \pi_{\lambda_*},
\Theta_{n-1,n-\lambda_1}(R^{\UU_{n-1}}_{\GG_{\ell-1} \times \UU_{m}}(\tau_2\otimes\pi_{\lambda'}))\otimes\omega_{n-\lambda_1}\rangle_{\UU_{n-\lambda_1}(\Fq)}\\
=&\langle \pi_{\lambda_*},
R^{\UU_{n-\lambda_1}}_{\GG_{\ell-1} \times \UU_{m-\lambda_1+1}}(\tau_2\otimes\Theta_{m,m-(\lambda_1-1)}(\pi_{\lambda'}))\otimes\omega_{n-\lambda_1}\rangle_{\UU_{n-\lambda_1}(\Fq)}.\\
=&\langle \pi_{\lambda_*},
R^{\UU_{n-\lambda_1}}_{\GG_{\ell-1} \times \UU_{m-\lambda_1+1}}(\tau_2\otimes\Theta_{m,m-\lambda_1'}(\pi_{\lambda'}))\otimes\omega_{n-\lambda_1}\rangle_{\UU_{n-\lambda_1}(\Fq)}\\
=&\langle \pi_{\lambda_*},
R^{\UU_{n-\lambda_1}}_{\GG_{\ell-1} \times \UU_{m-\lambda_1+1}}(\tau_2\otimes \pi_{\lambda'_*})\otimes\omega_{n-\lambda_1}\rangle_{\UU_{n-\lambda_1}(\Fq)}\
\end{aligned}
\]
Since the Weil representation of a finite unitary group is self-dual by \cite{Ger} ,  we see from  Proposition \ref{7.31} that the above last term  is equal to
 \[
 \langle \pi_{\lambda_*}\otimes\omega_{n-\lambda_1},
R^{\UU_{n-\lambda_1}}_{\GG_{\ell-1} \times \UU_{m-\lambda_1+1}}(\tau_2\otimes\pi_{\lambda'_*})\rangle_{\UU_{n-\lambda_1}(\Fq)}\\
= m(\lambda_*,\lambda'_*).
 \]
Since $\lambda_1\neq \lambda_1'$,  $\lambda$ and $\lambda'$ are 2-transverse if and only if $\lambda_*$ and $\lambda'_*$ are 2-transverse, which completes the proof by induction on $n$.
\end{proof}

Finally, Theorem \ref{1.1}
follows immediately from Proposition \ref{u1} and Theorem \ref{UU}.

\section{Generalization using Reeder's formula}\label{sec5}

Let $\pi$ and $\pi'$ be  representations of $\UU_n\fq$ and $\UU_m\fq$ respectively,  $n\ge m$. We have calculated $m(\pi,\pi')$ when $\pi$ and $\pi'$ are both unipotent. The goal of this section is  to prove Theorem \ref{1.2}, which extends the previous result when $\pi'$ is an arbitrary representation. We shall follow the method in our previous work \cite{LW}.

\subsection{Reeder's formula}
Let $G$ be a connected reductive algebraic
group over $\mathbb{F}_q$,  $H \subset G$ be a connected reductive subgroup of $G$ over $\Fq$, and $T$ and $S$ be $F$-stable maximal tori of $G$ and $H$ respectively.

In \cite{R}, Reeder gives a formula  for the multiplicity $\langle R_{T,\theta}^G,R_{S,\theta'}^H\rangle _{H^F}$ when $G$ and $H$ are simple. More precisely, by \cite [Theorem 1.4]{R} there is a polynomial $M(t)$ of degree at most $\delta$ whose coefficients depend on the characters $\theta$ and $\theta'$ of $T^F$ and $S^F$ respectively, and an integer $m \ge 1$ such that
\[
\langle R_{T,\theta^\nu}^G,R_{S,\theta^{\prime \nu}}^H\rangle _{H^{F^\nu}}=M(q^\nu)
\]
 for all positive integers $\nu \equiv 1$ mod $m$, where $\theta^\nu=\theta\circ N_\nu^T$ and $ N_\nu^T: T^{F^\nu}\to T^F$ is the norm map. The degree $\delta$ given in \cite{R} is optimal. Moreover, \cite[Proposition 7.4]{R} gives an explicit formula for the leading coefficient  in $M(t)$.  In order to calculate $\langle R^{\UU_{n+1}}_{L}(\tau\otimes\pi'),\pi_\lambda\rangle _{\UU _n(\Fq)}$ using Reeder's method, it is necessary to extend his result from connected simple algebraic groups to unitary groups. For the notations below we refer the readers to  \cite{R} and \cite{LW}, and from now on we put $(G^F,H^F)=(\UU_{n+1}\fq,\UU_{n}\fq)$. In {\it loc. cit.} we obtained the following:

 \begin{proposition}\label{prop5.1}
  Assume that $(G^F,H^F)=(\UU_{n+1}\fq,\UU_{n}\fq)$. Then
 \[
\langle R^G_{T,\chi},R^H_{S,\eta}\rangle _{H^F}=\sum_{\mbox{\tiny$\begin{array}{c}\iota\in I(S)^F\\
\delta_\iota=0\end{array}$}}\frac{(-1)^{\mathrm{rk}(G_\iota)+\mathrm{rk}(H_\iota)+\mathrm{rk}(T)+\mathrm{rk}(S)}}{|\bar{N}_{H}(\iota,S)^F|}\langle \chi_v, \eta_\varsigma\rangle _{Z_\iota^F}.
\]
where $v=j_{G_\iota}^{-1}(\mathrm{cl}(T,G))$ and $\varsigma=j_{H_\iota}^{-1}(\mathrm{cl}(S,H))$ for some $
\iota$ such that $j_{G_\iota}^{-1}(\mathrm{cl}(T,G))$ and $j_{H_\iota}^{-1}(\mathrm{cl}(S,H))$ are not empty.
\end{proposition}

Recall that for a semisimple element $s\in\UU_n\fq$, we say that $1\not\in s$ if 1 is not an eigenvalue of $s$. If a pair $(T,\theta)$ corresponds to $(T^*,s)$, then we say that $1\not\in (T,\theta)$ if $1\not\in s$. Then by Proposition 4.6 and Equation (5.3) in \cite{LW} which further explicate the above Reeder's formula   Proposition \ref{prop5.1}, we can easily deduce the following result.

 \begin{proposition}\label{proposition:A}
 Assume that $(G^F,H^F)=(\UU_{n+1}\fq,\UU_{n}\fq)$. Let $T_1\times T_2$ and $T_1'\times T_2$ be F-stable maximal tori of $G$. Assume that $1\notin (T_1,\theta)$ and $1\notin (T_1',\theta')$. Then
 \[
\varepsilon_{T_1}\langle R^G_{T_1\times T_2,\theta\otimes1},R^H_{S,1}\rangle _{H^F}=\varepsilon_{T_1'}\langle R^G_{T_1'\times T_2,\theta'\otimes1},R^H_{S,1}\rangle _{H^F}.
\]
\end{proposition}

Loosely speaking, this proposition says that for Deligne-Lusztig characters $\chi'$ and $\chi$ of $\UU_{n+1}\fq$ and   $\UU_{n}\fq$ respectively, if $\chi$ is unipotent, then to calculate $m(\chi, \chi')$ one only needs to consider the ``unipotent part" of $\chi'$.

Recall that for a semisimple element $s\in \UU_n\fq$,  $\pi_s^{reg}$ denotes  the unique irreducible regular character in $\mathcal{E}(\UU_n\fq,s)$. We have the following corollary of
Proposition \ref{proposition:A}.

 \begin{corollary}\label{5c1}
 Let $s_0$ and $s$ be two semisimple elements in $\UU_\ell\fq$ such that $1\not\in s_0, s$.   Assume that $s_0$ is regular so that $\pi^{reg}_{s_0}=\pm R^{\UU_\ell}_{T_0^*,s}$, where $T_0=C_{\UU_\ell}(s_0)$. Let $\pi$ be an irreducible representation of $\UU_\ell\fq$ in $\mathcal{E}(\UU_\ell\fq,s)$, and $\pi_{\lambda}$ and $\pi_{\lambda'}$ be  unipotent representations of $\UU_{n}\fq$ and $\UU_{n+1-\ell}\fq$ respectively. Then the following hold.

  (i) If
$
  \langle R^{\UU_{n+1}}_{\UU_\ell\times\UU_{n+1-\ell}}(\pi^{reg}_{s_0}\otimes \pi_{\lambda'}),\pi_{\lambda}\rangle_{\UU_{n}\fq}=0,
$
  then
$
  \langle R^{\UU_{n+1}}_{\UU_\ell\times\UU_{n+1-\ell}}(\pi\otimes \pi_{\lambda'}),\pi_{\lambda}\rangle_{\UU_{n}\fq}=0.
$

  (ii) If
$
  \langle R^{\UU_{n+1}}_{\UU_\ell\times\UU_{n+1-\ell}}(\pi^{reg}_{s_0}\otimes \pi_{\lambda'}),\pi_{\lambda}\rangle_{\UU_{n}\fq}=\pm 1,
$
  then
 \[
 \begin{aligned}
  \langle R^{\UU_{n+1}}_{\UU_\ell\times\UU_{n+1-\ell}}(\pi\otimes \pi_{\lambda'}),\pi_{\lambda}\rangle  _{\UU_{n}\fq}=
  \left\{
\begin{array}{ll}
\pm 1, &  \textrm{if }\ \pi =\pi_{s}^{reg},\\
0, & \textrm{otherwise.}
\end{array}\right.
\end{aligned}
  \]
\end{corollary}

\begin{proof} (i)
Since
$
\pi^{reg}_{s_0}=\pm R^{\UU_\ell}_{T_0^*,s_0},
$
we have
\[
R^{\UU_{n+1}}_{\UU_\ell\times\UU_{n+1-\ell}}(\pi^{reg}_{s_0}\otimes \pi_{\lambda'})=\pm\frac{1}{|W_{n+1-\ell}|}\sum_{ w\in W_{n+1-\ell}}\sigma_{\lambda'}(ww_0) R^{\UU_{n+1}}_{T_0^*\times T_w^*, (s_0,1)},
\]
and therefore
\[
\langle R^{\UU_{n+1}}_{\UU_\ell\times\UU_{n+1-\ell}}(\pi^{reg}_{s_0}\otimes \pi_{\lambda'}),\pi_{\lambda}\rangle_{\UU_{n}\fq}=\pm \frac{1}{|W_{n+1-\ell}|}\sum_{ w\in W_{n+1-\ell}}\sigma_{\lambda'}(ww_0) \langle R^{\UU_{n+1}}_{T_0^*\times T_w^*, (s_0,1)} ,\pi_{\lambda}\rangle_{\UU_{n}\fq}.
\]
Let us write $\pi\in \cal{E}(\UU_\ell\fq, s)$ as
\[
\pi=\sum_{ T^*\ni s} C_{T^*} R^{\UU_\ell}_{T^*,s}.
\]
Then similarly
\[
  \langle R^{\UU_{n+1}}_{\UU_\ell\times\UU_{n+1-\ell}}(\pi\otimes \pi_{\lambda'}),\pi_{\lambda}\rangle_{\UU_{n}\fq}= \frac{1}{|W_{n+1-\ell}|}\sum_{T^*\ni s}\sum_{ w\in W_{n+1-\ell}}C_{T^*}\sigma_{\lambda'}(ww_0) \langle R^{\UU_{n+1}}_{T^*\times T_w^*, (s,1)} ,\pi_{\lambda}\rangle_{\UU_{n}\fq}.
\]
By Proposition \ref{proposition:A}, up to sign the last term is  equal to
\begin{align}
&\frac{1}{|W_{n+1-\ell}|}\sum_{T^*\ni s}\sum_{ w\in W_{n+1-\ell}}\varepsilon_{T_0}\varepsilon_TC_{T^*}\sigma_{\lambda'}(ww_0)C_{T^*}\langle R^{\UU_{n+1}}_{T_0^*\times T_w^*,(s_0,1)},\pi_{\lambda}\rangle_{\UU_{n}\fq}\nonumber\\
=&\varepsilon_{T_0}\sum_{T^*\ni s}\varepsilon_TC_{T^*}\langle R^{\UU_{n+1}}_{\UU_\ell\times \UU_{n+1-\ell}}(\pi^{reg}_{s_0}\otimes \pi_{\lambda'}),\pi_{\lambda}\rangle_{\UU_{n}\fq}.\label{5.3-mult}
\end{align}
This finishes the proof of (i).

(ii) By (\ref{reg}) we have
\[
\langle \pi_{s}^{reg},R^{\UU_\ell}_{T^*,s}\rangle_{\UU_{\ell}\fq}=\varepsilon_{T}\varepsilon_{\UU_\ell}.
\]
Then up to sign we have
\[
\begin{aligned}
(\ref{5.3-mult}) & =\varepsilon_{T_0}\sum_{ T^*\ni s}\varepsilon_TC_{T^*}\\
& =\varepsilon_{T_0}\sum_{ T^*\ni s}\varepsilon_{T}C_{T^*}\cdot\varepsilon_{T}\varepsilon_{\UU_\ell}\langle \pi_{s}^{reg},R^{\UU_\ell}_{T^*,s}\rangle_{\UU_{\ell}\fq}\\
& =\varepsilon_{T_0}\varepsilon_{\UU_\ell}\langle \pi_{s}^{reg},\sum_{ T^*\ni s} C_{T^*} R^{\UU_\ell}_{T^*,s}\rangle_{\UU_{\ell}\fq},\\
& = \varepsilon_{T_0}\varepsilon_{\UU_\ell}\langle \pi_{s}^{reg}, \pi\rangle_{\UU_\ell\fq}=\left\{
\begin{array}{ll}
\pm 1, &  \textrm{if }\ \pi =\pi_{s}^{reg},\\
0, & \textrm{otherwise.}
\end{array}\right.
\end{aligned}
\]
The proof of (ii) is done.
\end{proof}

\subsection{Branching law for $\UU_n\fq$}
We shall keep the notations in Corollary \ref{5c1}. By this corollary, for the Bessel case we only need to calculate
  \[
  \langle R^{\UU_{n+1}}_{\UU_\ell\times\UU_{n+1-\ell}}(\pi^{reg}_{s_0}\otimes \pi_{\lambda'}),\pi_{\lambda}\rangle_{\UU_{n}\fq}
  \]
  where $s_0$ is regular semisimple. To this end we need the explicit theta lifting of $R^{\UU_{n+1}}_{\UU_\ell\times \UU_{n+1-\ell}}(\pi^{reg}_{s_0}\otimes \pi_{\lambda'})$, which was given in \cite[Theorem 2.6]{AMR}:

 \begin{proposition}\label{t2}
Let $s_0$ be a regular semisimple element of $\UU_\ell\fq$ such that $1\not\in s_0$.  Then (up to sign)
\[
\Theta_{n+1, n'} (R^{\UU_{n+1}}_{\UU_\ell\times\UU_{n+1-\ell}}(\pi^{reg}_{s_0}\otimes \pi_{\lambda'}))= R^{\UU_{n'}}_{\UU_\ell\times\UU_{n'-\ell}}(\pi^{reg}_{s_0}\otimes\Theta_{n+1-\ell, n'-\ell}( \pi_{\lambda'})).
\]
\end{proposition}

As a consequence, we obtain the following extension of Theorem \ref{UU} using the same see-saw arguments.

 \begin{proposition}\label{UU2}
Let $s_0$ be a regular semisimple element of $\UU_\ell\fq$ such that $1\not\in s_0$, and let $\lambda'$ and $\lambda''$ be partitions of $n+1-\ell$ and $n-\ell$ respectively. Then

(i)
$
 \langle R^{\UU_{n+1}}_{\UU_\ell\times\UU_{n+1-\ell}}(\pi^{reg}_{s_0}\otimes \pi_{\lambda'}),\pi_{\lambda}\rangle_{\UU_{n}\fq}=\left\{
\begin{array}{ll}
\pm1, &  \textrm{if } \lambda \textrm{ and } \lambda' \textrm{ are }  2\textrm{-transverse,}\\
0, & \textrm{otherwise.}
\end{array}\right.
$

(ii)
$
 \langle R^{\UU_{n}}_{\UU_\ell\times\UU_{n-\ell}}(\pi^{reg}_{s_0}\otimes \pi_{\lambda''}),\pi_{\lambda}\otimes \omega_{n}\rangle_{\UU_{n}\fq}=\left\{
\begin{array}{ll}
\pm1, &  \textrm{if } \lambda \textrm{ and } \lambda'' \textrm{ are }  2\textrm{-transverse,}\\
0, & \textrm{otherwise.}
\end{array}\right.
$
\end{proposition}
\begin{proof}
We may prove proposition by induction on $n$. If $n=1$, then by Proposition \ref{proposition:A},
\[
\langle{\bf 1},\pm R^{\UU_2}_{\GG_1,\theta}\rangle_{\UU_1(\Fq)}=1, \quad \langle{\bf 1},\pm R^{\UU_2}_{\UU_1\times\UU_1},\theta'\otimes 1\rangle_{\UU_1(\Fq)}=0 \quad \textrm{and}\quad \langle {\bf 1}\otimes\omega_{1}, {\bf 1}\rangle_{\UU_1(\Fq)}=0,
\]
where $\theta$ is regular and $\theta'\neq 1$. Assume that the proposition holds for $n'<n$.  To finish the induction, one only needs to apply Proposition \ref{t2} instead of Proposition \ref{theta} to calculate the theta lifting in the proof of Theorem \ref{UU}. The rest of the proof is similar  and
will be left to the reader.
\end{proof}

Finally we are ready to prove Theorem \ref{1.2}. For convenience let us recall its statement:

 \begin{theorem}\label{UU3}
Let $\lambda$ and $\lambda'$ be partitions of $n$ and $m$ respectively, $m\leq n$.
Let $\pi\in \mathcal{E}(\UU_{\ell}\fq, s)$ with $\ell+m\le n+1$ and $1\not\in s$. Then
\[
m(\pi_{\lambda},R_{\UU_{\ell}\times\UU_m}^{\UU_{\ell+m}}(\pi\otimes \pi_{\lambda'}) )=\left\{
\begin{array}{ll}
1, &  \textrm{if }\ \lambda \textrm{ and } \lambda' \textrm{ are }  2\textrm{-transverse and } \pi=\pi^{reg}_s,\\
0, & \textrm{otherwise,}
\end{array}\right.
\]
where  $\pi_s^{reg}$ is the unique regular character in $\mathcal{E}(\UU_{\ell}\fq,s)$.
\end{theorem}

\begin{proof} Put $\ell_0:=\lceil (n+1-\ell-m)/2\rceil$.
For the Bessel case, by \cite[Proposition 5.2]{LW} and Proposition \ref{3.1}, we only need to compute
\begin{equation}\label{mult}
\langle R^{\UU_{n+1}}_{\UU_{n+1-m}\times \UU_m}(R^{\UU_{n+1-m}}_{\GG_{\ell_0}\times\UU_{\ell}}(\tau\otimes\pi)\otimes \pi_{\lambda'}),\pi_{\lambda}\rangle_{\UU_{n}\fq},
\end{equation}
where $\tau$ is an irreducible cuspidal representation of $\GG_{\ell_0}\fq$ that is not conjugate self-dual. Then
$\tau=\pm R^{\GG_{\ell_0}}_{T, t}$ for some regular semisimple element $t$ of $\GG_{\ell_0}\fq$ such that $T:=C_{\GG_{\ell_0}}(t)$ is minisotropic.   Let $t'$ be the image of $t$ in $\UU_{2\ell_0}$. Note that $1\not\in t'$ and $t'$ is in fact regular in $\UU_{2\ell_0}$. Our assumption on the finite field $\bb{F}_q$ implies that we may choose $\tau$ such that $t'$ and $s$ have no common eigenvalues. Then by Proposition \ref{3.1} and Proposition \ref{irr},
\[
R^{\UU_{n+1-m}}_{\GG_{\ell_0}\times\UU_{\ell}}(\tau\otimes\pi)=R^{\UU_{n+1-m}}_{\UU_{2\ell_0}\times\UU_{\ell}}\left((R^{\UU_{2\ell_0}}_{\GG_{\ell_0}}\tau)\otimes\pi\right)
\]
is irreducible.
By Proposition \ref{UU2} and Corollary \ref{5c1},
\[
(\ref{mult})=\left\{
\begin{array}{ll}
\pm1, &  \textrm{if }\ \lambda \textrm{ and } \lambda' \textrm{ are }  2\textrm{-transverse and } R^{\UU_{n+1-m}}_{\GG_{\ell_0}\times\UU_{\ell}}(\tau\otimes\pi)=\pi^{reg}_{(t',s)}.\\
0, & \textrm{otherwise.}
\end{array}\right.
\]
By Proposition \ref{Lus} and (\ref{reg}), $R^{\UU_{n+1-m}}_{\GG_{\ell_0}\times\UU_{\ell}}(\tau\otimes\pi)$ is regular if and only if $\pi$ is regular, which completes the proof.

We now turn to the Fourier-Jacobi case. By Proposition \ref{7.31}, we only need to compute
\begin{equation}\label{mult-2}
\langle R^{\UU_{n}}_{\UU_{n-m}\times\UU_m}(R^{\UU_{n-m}}_{\GG_{\ell_0}\times\UU_{\ell}}(\tau\otimes\pi)\otimes \pi_{\lambda'}),\pi_{\lambda}\otimes\omega_{n}\rangle_{\UU_{n}\fq},
\end{equation}
where $\tau$ is as above.
By Proposition \ref{irr} again,
 \[
R^{\UU_{n}}_{\UU_{n-m}\times\UU_m}(R^{\UU_{n-m}}_{\GG_{\ell_0}\times\UU_{\ell}}(\tau\otimes\pi)\otimes \pi_{\lambda'})
\]
 is irreducible.
Put $\mu:=[\mu_1,\lambda]$, where $\mu_1>\lambda_1$. By Corollary \ref{CU1} (ii),
\[
\Theta_{|\mu|,n}(\pi_\mu)=\pi_{\lambda}.
\]
Consider the see-saw diagram
\[
\setlength{\unitlength}{0.8cm}
\begin{picture}(20,5)
\thicklines
\put(6.8,4){$\UU_{n}\times \UU_{n}$}
\put(7.5,1){$\UU_{n}$}
\put(12.3,4){$\UU_{|\mu|+1}$}
\put(12.0,1){$\UU_{|\mu|}\times \UU_1$}
\put(7.7,1.5){\line(0,1){2.1}}
\put(12.8,1.5){\line(0,1){2.1}}
\put(8,1.5){\line(2,1){4.2}}
\put(8,3.7){\line(2,-1){4.2}}
\end{picture}
\]
By the see-saw identity, one has
\begin{align*}
(\ref{mult-2})& =\langle R^{\UU_{n}}_{\UU_{n-m}\times\UU_m}(R^{\UU_{n-m}}_{\GG_{\ell_0}\times\UU_{\ell}}(\tau\otimes\pi)\otimes \pi_{\lambda'}),\Theta_{|\mu|,n}(\pi_\mu)\otimes\omega_{n}\rangle_{\UU_{n}\fq}\\
& = \langle\Theta_{n,|\mu|+1}(  R^{\UU_{n}}_{\UU_{n-m}\times\UU_m}(R^{\UU_{n-m}}_{\GG_{\ell_0}\times\UU_{\ell}}(\tau\otimes\pi)\otimes \pi_{\lambda'})),\pi_\mu\rangle_{\UU_{n}\fq}\\
& = \langle R^{\UU_{|\mu|+1}}_{\UU_{n-m}\times\UU_{|\mu|+1-n+m}}(R^{\UU_{n-m}}_{\GG_{\ell_0}\times\UU_{\ell}}(\tau\otimes\pi)\otimes\Theta_{m, |\mu|+1-n+m} (\pi_{\lambda'})),\pi_\mu\rangle_{\UU_{n}\fq}.\\
\end{align*}
The rest follows from our proof for the Bessel case.
\end{proof}


\begin{thebibliography}{CERP}

\bibitem [A]{A}
D. Alvis, {\it  The duality operation in the character ring of a finite Chevalley group,} Bull. Amer. Math. Soc. (N.S.) {\bf 1} (1979), 907--911.

\bibitem[AM]{AM}
J. Adams, A. Moy, {\it Unipotent representations and reductive dual pairs over finite fields}, Trans. Amer.
Math. Soc. {\bf 340} (1993), 309--321.

%\bibitem[AGRS]{AGRS}
%A. Aizenbud, D. Gourevitch, S. Rallis, G. Schiffmann,
%{\it Multiplicity one theorems},
%Ann. of Math. (2) {\bf 172} (2010), no. 2, 1407--1434.

\bibitem[Ato]{Ato}
H. Atobe, \textit{The local theta correspondence and the local Gan-Gross-Prasad conjecture for the symplectic-metaplectic case.}  Math. Ann. {\bf 371} (2018), no. 1-2, 225--295.

\bibitem[AMR]{AMR}
A.-M. Aubert, J. Michel, R. Rouquier, {\it Correspondance de Howe pour les groupes reductifs sur les corps finis}, Duke Math. J. {\bf83}, 2 (1996), 353--397.

\bibitem[BP1]{BP1}
R. Beuzart-Plessis, \textit{La conjecture locale de Gross-Prasad pour les repr\'esentations temp\'er\'ees des groupes unitaires.} M\'em. Soc. Math. Fr. (N.S.) 2016, no. 149, vii+191 pp.

\bibitem[BP2]{BP2}
\bysame, \textit{Endoscopie et conjecture locale raffin\'ee de Gan-Gross-Prasad pour les groupes unitaires,}  Compos. Math. {\bf 151} (2015), no. 7, 1309--1371.

\bibitem[C]{C}
R. Carter, {\it Finite Groups of Lie Type, Conjugacy Classes and Complex Characters}, John Wiley $\&$ Sons, England, 1985.

\bibitem[Cu]{Cu}
C.W. Curtis, {\it Truncation and duality in the character ring of a finite group of Lie type,} J.Algebra {\bf 62} (1980), 320šC332.

\bibitem[DL]{DL}
P. Deligne, G. Lusztig, {\it Representations of reductive groups over finite fields}, Ann. of Math. {\bf 103} (1976), 103--161.

\bibitem[GGP1]{GGP1}
W. T. Gan, Benedict H. Gross and D. Prasad, {\it Symplectic local root numbers, central critical L-values and restriction problems in the representation theory of classical groups}, Ast\'erisque. No. {\bf346} (2012), 1--109.

\bibitem[GGP2]{GGP2}
\bysame, {\it Restrictions of  representations  of  classical  groups: examples}, Ast\'erisque. No.{\bf346} (2012), 111--170.

\bibitem[GI]{GI}
W. T. Gan, A. Ichino, \textit{The Gross-Prasad conjecture and local theta correspondence.} Invent. Math. {\bf 206} (2016), no. 3, 705--799.

\bibitem[Ger]{Ger}
P. G\'erardin, {\it Weil representations associated to finite fields}, J. Algebra {\bf 46} (1977), 54--101.

%\bibitem[GRS]{GRS}
%D. Ginzburg, S. Rallis, D. Soudry, {\it The descent map from automorphic representations of $\GGL(n)$ to classical groups.} World Scientific Publishing Co. Pte. Ltd., Hackensack, NJ, 2011. x+339 pp.

\bibitem[GP1]{GP1}
B. Gross, D. Prasad, \textit{On the decomposition of a representation of $\mathrm{SO}_n$ when restricted to $\mathrm{SO}_{n-1}$.} Canad. J. Math.
44 (1992), 974--1002.

\bibitem[GP2]{GP2}
\bysame, \textit{On irreducible representations of $\mathrm{SO}_{2n+1}\times\mathrm{SO}_{2m}$.} Canad. J. Math. 46 (1994), 930--950.

\bibitem[HZ]{HZ}
G. Hiss, A. Zalesski, {\it The Weil-Steinberg character of finite classical groups}, Representation theory, {\bf13} (2009), 427-459.

\bibitem[JZ1]{JZ1}
D. Jiang, L. Zhang, {\it Local root numbers and spectrum of the local descents for orthogonal groups: $p$-adic case,}  Algebra Number Theory  {\bf 12} (2018), no. 6, 1489--1535.


\bibitem[K]{K}
N. Kawanaka, {\it Fourier transforms of nilpotently supported invariant functions on a finite simple Lie algebra,} Proc. Japan Acad. Ser. A Math. Sci. {\bf 57} (1981), 461--464.

\bibitem[L]{L}
G. Lusztig, {\it Irreducible representations of finite classical groups}, Invent. Math. {\bf43} (1977), 125--175.

%\bibitem[L2]{L2}
%\bysame, {\it Characters of reductive groups over a finite field}, Princeton Univ. Press, Princeton, N.J., 1984.



\bibitem[LS]{LS}
G. Lusztig, B. Srinivasan, {\it The characters of the finite unitary groups}, J. Algebra {\bf49} (1977), 167-171.

\bibitem[LW]{LW}
D. Liu, Z. Wang, {\it Descents of unipotent representations of finite unitary groups}, Trans. Amer. Math. Soc. {\bf 373} (2020), 4223--4253.

\bibitem[MVW]{MVW}
C. M\oe glin, M.-F. Vign\'eras, J.-L. Waldspurger, \textit{Correspondances de Howe sur un
corps $p$-adique}, Springer Verlag, Lecture Notes in Math. {\bf 1291}, Berlin, Heidelberg, 1987.

\bibitem[MW]{MW}
C. M\oe glin, J.-L. Waldspurger,  \textit{La conjecture locale de Gross-Prasad pour les groupes sp\'eciaux orthogonaux: le cas g\'en\'eral.} Sur
les conjectures de Gross et Prasad. II. Ast\'erisque No. {\bf 347} (2012),
167--216.

%\bibitem[P]{P}

%S.-Y. Pan, {\it Local theta correspondence of depth zero representations and theta dichotomy}, J. Math. Soc. Japan, Vol. {\bf54}, No. 4, (2002) 794--845.



\bibitem[R]{R}
M. Reeder, {\it On the restriction of Deligne-Lusztig characters}, J. Amer. Math. Soc. {\bf20} (2007) 573--602.



\bibitem[S1]{S1}
B. Srinivasan, {\it Representations of finite Chevalley groups}, Lecture Notes in Math. {\bf 764}, Springer-Verlag, 1979.

\bibitem[S2]{S2}
\bysame, {\it Weil representations of finite classical groups}, Invent. Math. {\bf 51} (1979), 143--153.

%\bibitem[Su]{Su}
%B. Sun, {\it Multiplicity one theorems for Fourier-Jacobi models,}
%Amer. J. Math. {\bf 134} (2012), no. 6, 1655--1678.

\bibitem[W1]{W1}
J.-L. Waldspurger,  \textit{Une formule int\'egral \`a la conjecture locale de Gross-Prasad.} Compos. Math. {\bf 146} (2010), no. 5, 1180--1290.

\bibitem[W2]{W2}
\bysame, \textit{Une formule int\'egrale reli\'ee \`a la conjecture
locale de Gross-Prasad, 2e partie: extension aux repr\'esentations
temp\'er\'ees.} Sur les conjectures de Gross et Prasad. I. Ast\'erisque
No. {\bf 346} (2012), 171--312.

\bibitem[W3]{W3}
\bysame, \textit{La conjecture locale de Gross-Prasad pour les
repr\'esentations temp\'er\'ees. des groupes sp\'eciaux orthogonaux.} Sur
les conjectures de Gross et Prasad. II. Ast\'erisque No. {\bf 347} (2012),
103--165.

%\bibitem[Y]{Y}
%J.-K. Yu, {\it Construction of tame supercuspidal representations,} J. Amer. Math. Soc. {\bf 14} (2001), no. 3, 579--622.


\end{thebibliography}
\end{document}